\documentclass[12pt,reqno]{amsart}
\usepackage{latexsym,amsmath,amssymb,amsfonts,amscd,graphics,appendix,amsxtra}
\usepackage{mathrsfs}
\usepackage{graphicx}
\usepackage[normalem]{ulem}
\usepackage{comment}
\usepackage{makecell}
\usepackage[utf8]{inputenc}
\usepackage{tikz}
\usetikzlibrary{matrix,arrows,decorations.pathmorphing}
\usepackage{tikz-cd}
\usepackage[all]{xy}
\usepackage{fullpage}
\usepackage{float}
\floatstyle{plaintop}
\restylefloat{table}
\usepackage{booktabs}
\usepackage{array}
\usepackage{hyperref}
\newcolumntype{P}[1]{>{\centering\arraybackslash}p{#1}}

\newtheorem{theorem}{Theorem}[section]
\newtheorem{lemma}[theorem]{Lemma}
\newtheorem{proposition}[theorem]{Proposition}

\theoremstyle{definition}
\newtheorem{remark}[theorem]{Remark}

\newtheorem{definition}[theorem]{Definition}

\numberwithin{equation}{section}

\newcommand{\Bl}{\mathrm{Bl}}

\hypersetup{
    colorlinks=true,
    linkcolor=red,
    filecolor=magenta,      
    urlcolor=red,
}

\title{An explicit wall crossing for the moduli space of hyperplane arrangements}

\author{Patricio Gallardo}
\address{Department of Mathematics, University of California, Riverside, CA 92521, USA }
\email{pgallard@ucr.edu}

\author{Luca Schaffler}
\address{Dipartimento di Matematica e Fisica, Universit\`a degli Studi Roma Tre, Largo San Leonardo Murialdo 1, 00146, Roma, Italy}
\email{luca.schaffler@uniroma3.it}


\begin{document}

\begin{abstract}
The moduli space of hyperplanes in projective space has a family of geometric and modular compactifications that parametrize stable hyperplane arrangements with respect to a weight vector. Among these, there is a toric compactification that generalizes the Losev--Manin moduli space of points on the line. We study the first natural wall crossing that modifies this compactification into a non-toric one by varying the weights. As an application of our work, we show that any $\mathbb{Q}$-factorialization of the blow up at the identity of the torus of the generalized Losev--Manin space is not a Mori dream space for a sufficiently high number of hyperplanes. Additionally, for lines in the plane, we provide a precise description of the wall crossing.
\end{abstract}

\maketitle


\section{Introduction}
Understanding the global geometry of the KSBA moduli spaces for higher-dimensional pairs is a notoriously difficult problem. A notable case that has recently been understood \cite{AAB24} involves those associated with stable log Calabi--Yau surfaces, where every irreducible component of the coarse moduli space is a finite quotient of a toric variety. Furthermore, the fan of this toric variety is constructed via mirror symmetry. Other examples include the foundational work on stable toric pairs and abelian varieties \cite{Ale02}, and other explicit examples for moduli of surfaces such as \cite{MS21}, \cite{DH21}, \cite{DeV22}, \cite{AET23}, among others. Given an explicit description of a KSBA moduli space, a natural question arises: how does the space change when the parametrized pairs become more positive by adjusting their weights? According to general theory \cite{ABIP23}, a wall-crossing phenomenon exists, and the moduli space is modified. Yet, little is known about the explicit geometry of such modifications, and examples for moduli of higher dimensional varieties as in \cite{Sch23} are scarce.

Our work is a step toward addressing this question by starting with the moduli of weighted stable hyperplane arrangements $\overline{\mathrm{M}}_{\mathbf{b}}(\mathbb{P}^d,n)$ with respect to a weight vector $\mathbf{b}$, see \cite{Ale08,Ale15} and \cite{HKT06} for the case $\mathbf{b}=(1,\ldots,1)$. There is a specific choice of weight vector $\mathbf{t}$ (see Definition~\ref{def:weight-vector-t}) such that the main component of the KSBA moduli space is the toric variety associated with a fiber polytope. This space is known as the (generalized) Losev--Manin moduli space of hyperplane arrangements and we denote it by $\overline{\mathrm{LM}}(\mathbb{P}^d,n)$. By changing the weights of the hyperplanes to $\mathbf{nt}$ as in Definition~\ref{def:weight_nt}, we make the pairs more positive, and the resulting KSBA moduli space \( \overline{\mathrm{M}}_{\mathbf{nt}}(\mathbb{P}^d,n)\) is a non-toric modification of $\overline{\mathrm{LM}}(\mathbb{P}^d,n)$.

\begin{theorem}[Theorem~\ref{thm:locally-at-e-is-simple-blow-up} and Theorem~\ref{thm:geometric-meaning-blow-up-identity-dim2}]
\label{thm:mainA}
There exists a surjective birational morphism
\[
\overline{\mathrm{M}}_{\mathbf{nt}}(\mathbb{P}^d,n)\rightarrow\overline{\mathrm{LM}}(\mathbb{P}^d,n),
\]
which, in a neighborhood of the identity $e$ of the dense open subtorus of $\overline{\mathrm{LM}}(\mathbb{P}^d,n)$, corresponds to the simple blow up of $e$, up to normalization. In particular, for $d=2$, we have an isomorphism
\[
\overline{\mathrm{M}}_{\mathbf{nt}}(\mathbb{P}^2,n)^\nu\cong\Bl_e\overline{\mathrm{LM}}(\mathbb{P}^2,n),
\]
where for a reduced scheme $X$ we let $X^\nu$ be its normalization.
\end{theorem}

We prove the above theorem in \S\,\ref{sec:wall-crossing} and \S\,\ref{sec:iso-in-case-of-lines}. The main idea is to study the stability of the pairs parametrized by $\overline{\mathrm{LM}}(\mathbb{P}^d,n)$ when we change the weight vector from $\mathbf{t}$ to $\mathbf{nt}$. Then, we explicitly calculate the stable replacements whenever we obtain unstable pairs. Care must be exercised because the log canonical divisor of certain pairs becomes not $\mathbb{Q}$-Cartier as we change the weights.
We conjecture that the analogous isomorphism holds for $d\geq3$. However, in this higher dimensional case, many more subtleties appear. For instance, the irreducible components of the degenerations parametrized by $\overline{\mathrm{LM}}(\mathbb{P}^d,n)$ may have non $\mathbb{Q}$-factorial singularities. This is because a mixed subdivision of a scaled simplex $\Delta_d$ for $d\geq3$ may involve singular polytopes. An example can be found in \cite[Figure~5]{CHSW11}. In any case, we expect $\overline{\mathrm{M}}_{\mathbf{nt}}(\mathbb{P}^d,n)^\nu$ to be isomorphic to $\overline{\mathrm{LM}}(\mathbb{P}^d,n)$ for any $d\geq3$. This is object of current investigation.

Next, we describe an application of the Theorem~\ref{thm:mainA}. A central question in algebraic geometry is to understand the birational geometry of varieties. 
We recall that a variety is a \emph{Mori dream space} if its birational geometry resembles that of a toric variety, see \S\,\ref{sec:MDS}.
It is difficult to determine if a variety is a Mori dream space, and the question is, in general, unanswered even for the blow up at the identity $e$ of the dense open subtorus of a projective smooth toric variety, see \cite{GK16}, \cite{HKL18}, \cite{He19}, \cite{GGK21}, \cite{GGR22}, and \cite{Cas18}. In our second result, we show the failure of the Mori dream property for the blow up at $e$ of $\overline{\mathrm{LM}}(\mathbb{P}^d,n)$. This conclusion stems directly from combining Theorem~\ref{thm:mainA}, standard tools within moduli theory, and the analogous failure for the Losev--Manin moduli space $\overline{\mathrm{LM}}(\mathbb{P}^1,n)$, as seen in \cite{CT15}, and subsequently in \cite{GK16} and \cite{HKL18}.

\begin{theorem}[Theorem~\ref{thm:Q-fact-not-MDS}]
Let $n\geq d+9$. Then any $\mathbb{Q}$-factorialization 
of $\Bl_e\overline{\mathrm{LM}}(\mathbb{P}^d,n)$ is not a Mori dream space.    
\end{theorem}

We prove the above theorem in \S\,\ref{sec:MDS}.  The main idea is to construct a dominant morphism $\Bl_e\overline{\mathrm{LM}}(\mathbb{P}^d,n)\rightarrow\Bl_e\overline{\mathrm{LM}}(\mathbb{P}^1,n-d+1)$ and then apply \cite{Oka16}. We point out that the existence of such morphism does not immediately follow by simply restricting to hyperplanes because the pairs have weights and $\Bl_e\overline{\mathrm{LM}}(\mathbb{P}^d,n)$ lacks a modular interpretation for $d \geq 3$ at this point. Furthermore, there is a surjective morphism $\alpha\colon\overline{\mathrm{LM}}(\mathbb{P}^d,n)\rightarrow\overline{\mathrm{LM}}(\mathbb{P}^1,n-d+1)$ such that $\alpha(e)=e$, but $\{e\}\subsetneq\alpha^{-1}(e)$. Thus, $\alpha$ does not automatically lift to the blow ups at $e$ by the universal property of the blow up.


\subsection*{Acknowledgements}
The authors would like to thank Valery Alexeev for introducing us to the problem and for helpful feedback on preliminary work. We also thank Angelo Felice Lopez and Filippo Viviani for useful comments, and Xian Wu for pointing out an imprecision in Lemma~3.6~(2) of the first draft of the paper. P. Gallardo is partially supported by the National Science Foundation under Grant No. DMS-2316749.  L. Schaffler is a member of the INdAM group GNSAGA and was supported by the projects ``Programma per Giovani Ricercatori Rita Levi Montalcini'', PRIN2017SSNZAW ``Advances in Moduli Theory and Birational Classification'', PRIN2020KKWT53 ``Curves, Ricci flat Varieties and their Interactions'', and PRIN 2022 ``Moduli Spaces and Birational Geometry'' -- CUP E53D23005790006.

\tableofcontents

\section{Preliminaries on weighted stable hyperplane arrangements}
\label{sec:comp-mod-weighted-shas}

We review the necessary results from \cite{Ale08,Ale15} and set up the notation. Let $X$ be a variety, $B_1,\ldots,B_n\subseteq X$ reduced divisors, and $0<b_1,\ldots,b_n\leq1$ rational numbers. We adopt the following notation:
\[
\mathbf{b}B:=\sum_{i=1}^nb_iB_i.
\]
The pair $(X,\mathbf{b}B)$ is called \emph{stable} provided it has \emph{semi-log canonical singularities} \cite[Definition--Lemma~5.10]{Kol13} and $K_X+\mathbf{b}B$ is ample.

We are interested in a particular type of stable pairs which we describe next.
Let $H_1,\ldots,H_n$ be hyperplanes in $\mathbb{P}^d$ and let $\mathbf{b}=(b_1,\ldots,b_n)$ be a vector of rational numbers satisfying $0<b_i\leq1$ for all $i$. Assume that $\sum_{i=1}^nb_i>d+1$ and that $(\mathbb{P}^d,\mathbf{b}H)$ is log canonical (the latter is satisfied, for instance, when the hyperplanes are linearly general). Then, $(\mathbb{P}^d,\mathbf{b}H)$ is a stable pair, and it is referred to as a \emph{weighted stable hyperplane arrangement}.

The theory of stable pairs and stable toric varieties induces a geometric and modular compactification of the moduli space of weighted stable hyperplane arrangements. 
\begin{theorem}[{\cite[Theorem~1.1]{Ale08}}]
\label{thm:family-over-moduli-weighted-hyperplane-arrangements}
Let $d\geq1, n\geq d+3$, and choose rational weights $\mathbf{b}=(b_1,\ldots,b_n)$ such that $0<b_i\leq1$ and $\sum_{i=1}^nb_i>d+1$. Then there exists a projective scheme 
$\overline{\mathrm{M}}_{\mathbf{b}}(\mathbb{P}^d,n)$ together with a flat and projective family $f\colon\left(\mathcal{X},\mathcal{B}_1,\ldots,\mathcal{B}_n\right)\rightarrow\overline{\mathrm{M}}_{\mathbf{b}}(\mathbb{P}^d,n)$ such that the following hold:
\begin{enumerate}

\item Every geometric fiber of $f$ is a $d$-dimensional variety $X$ together with Weil divisors $B_1,\ldots,B_n$ such that the pair $(X,\mathbf{b}B)$ is stable.

\item For distinct geometric points of $\overline{\mathrm{M}}_{\mathbf{b}}(\mathbb{P}^d,n)$ the fibers are non-isomorphic.

\item Over an open, but not necessarily dense, subset $\mathrm{M}_{\mathbf{b}}(\mathbb{P}^d,n)\subseteq\overline{\mathrm{M}}_{\mathbf{b}}(\mathbb{P}^d,n)$, the family $f$ coincides with the universal family of weighted stable arrangements of $n$ hyperplanes in projective space $\mathbb{P}^d$.

\end{enumerate}
For every positive integer $m$ such that all $mb_i$ are integers, the sheaf $\mathcal{O}_{\mathcal{X}}(m(K_{\mathcal{X}}+\sum_{i=1}^nb_i\mathcal{B}_i))$ is relatively ample and free over $\overline{\mathrm{M}}_{\mathbf{b}}(\mathbb{P}^d,n)$.
\end{theorem}

For the rest of the paper, we abuse our notation and write $\overline{\mathrm{M}}_{\mathbf{b}}(\mathbb{P}^d,n)$ to denote the irreducible component containing $\mathrm{M}_{\mathbf{b}}(\mathbb{P}^d,n)$ as a dense open subset. We pinpoint some special cases of the above construction. The moduli space $\overline{\mathrm{M}}_{\mathbf{b}}(\mathbb{P}^1,n)$ equals Hassett's moduli space $\overline{\mathrm{M}}_{0,\mathbf{b}}$ of $\mathbf{b}$-weighted stable rational curves in \cite{Has03}. In particular, if $\mathbf{b} = (1,1, \epsilon, \ldots, \epsilon)$, then corresponding Hassett's moduli space recovers the Losev--Manin moduli space \cite{LM00} which parametrizes marked chains of $\mathbb{P}^1$'s. If $\mathbf{1}:=(1,\ldots,1)$, then $\overline{\mathrm{M}}_\mathbf{1}(\mathbb{P}^d,n)$ recovers Kapranov's Chow quotient compactification of the moduli space of $n$ linearly general hyperplanes in $\mathbb{P}^d$ \cite{Kap93}, which was shown to be a moduli space for stable pairs in \cite{HKT06}. In the current paper, the following choice of weights will be of particular interest.

\begin{definition}
\label{def:weight-vector-t}
Let us define the weight vector
\[
\mathbf{t}=\mathbf{t}(d,n):=(\underbrace{1,\ldots,1}_{(d+1)\textrm{-times}},\epsilon,\ldots,\epsilon)\in\mathbb{Q}^n,
\]
where $0<\epsilon\ll1$. In analogy with the case of points in $\mathbb{P}^1$, we refer to $\overline{\mathrm{LM}}(\mathbb{P}^d,n):=\overline{\mathrm{M}}_{\mathbf{t}}(\mathbb{P}^d,n)$ as the \emph{Losev--Manin moduli space for hyperplanes in $\mathbb{P}^d$}. When simply writing $\mathbf{t}$, the values of $d$ and $n$ will be clear from the context.
\end{definition}

A fundamental aspect of such Losev--Manin spaces is that for all $d \geq 1$ and $n \geq d+3$,  $\overline{\mathrm{LM}}(\mathbb{P}^d,n)$ is a toric variety. For $d=1$, this was proved originally in \cite{LM00}.

\begin{theorem}[{\cite[Example~9.6]{Ale08}}]
\label{thm:LM-space-for-lines-is-toric}
Let $\Delta_d$ denote the $d$-dimensional simplex and let $\pi\colon\Delta_d^{n-d-1}\rightarrow(n-d-1)\Delta_d$ be the polytope fibration given by
\begin{align*}
\pi\colon\Delta_d^{n-d-1}&\rightarrow(n-d-1)\Delta_d,\\
(x_1,\ldots,x_{n-d-1})&\mapsto x_1+\ldots+x_{n-d-1}.
\end{align*}
Let $\Sigma(n):=\Sigma(\pi)$ be the corresponding \emph{fiber polytope} (we refer to \cite{BS92} for the definition). Then the toric variety $X_{\Sigma(n)}$ associated to $\Sigma(n)$ is isomorphic to the Losev--Manin moduli space $\overline{\mathrm{LM}}(\mathbb{P}^d,n)$.
\end{theorem}

As expected, the compact moduli spaces $\overline{\mathrm{M}}_{\mathbf{b}}(\mathbb{P}^d,n)$ depend on the weight vector $\mathbf{b}$. 

\begin{definition}
The collection of all possible weight vectors
\[
\mathcal{D}(d+1,n):=\left\{(b_1,\ldots,b_n)\in\mathbb{Q}^n\mid0<b_1,\ldots,b_n\leq1,~\sum_{i=1}^nb_i>d+1\right\}
\]
is called the \emph{weight domain}. A \emph{wall} is a hyperplane in $\mathbb{Q}^n$ in the form
\[
x_I:=\sum_{i\in I}x_i=k
\]
for some $I\subseteq\{1,\ldots,n\}$ with $2\leq|I|\leq n-2$ and $1\leq k\leq d$. The walls induce the so-called \emph{chamber decomposition} of the weight domain $\mathcal{D}(d+1,n)$.
\end{definition}

\begin{theorem}[{\cite[Theorem~5.5.2]{Ale15}}]
\label{thm:wall_crossing_valery}
The domain \(\mathcal{D}(d+1,n) \)  is divided by the hyperplanes $x_I = k $ for all $I \subseteq \{1, \ldots, n \}$, $2 \leq |I| \leq (n-2)$, $ 1\leq k \leq d$, into finitely many chambers. Then, 
\begin{enumerate}
    \item If $\mathbf{b}$ and $\mathbf{b}'$ lie in the same chamber, denoted $\mathbf{b} \sim \mathbf{b}'$, then 
    $\overline{\mathrm{M}}_{\mathbf{b}}(\mathbb{P}^d,n) =  
    \overline{\mathrm{M}}_{\mathbf{b}'}(\mathbb{P}^d,n)$, 
    and their families are the same;
    \item If $\mathbf{b}$ is in the closure of the chamber of $\mathbf{b}'$, denoted as $\mathbf{b} \in \overline{\mathbf{b}'}$, then there exist a contraction 
    $\overline{\mathrm{M}}_{\mathbf{b}'}(\mathbb{P}^d,n) \rightarrow 
    \overline{\mathrm{M}}_{\mathbf{b}}(\mathbb{P}^d,n)$ on the moduli spaces and a morphism between the families, that is, 
    $ \left( \mathcal{X}', \mathcal{B}' \right) \rightarrow 
      \left( \mathcal{X}, \mathcal{B} \right) $; 
      \item further, if $\mathbf{b} \in \overline{\mathbf{b'}}$ and 
      $\mathbf{b}' \leq \mathbf{b}$, then 
    $\overline{\mathrm{M}}_{\mathbf{b}}(\mathbb{P}^d,n) =  
    \overline{\mathrm{M}}_{\mathbf{b}'}(\mathbb{P}^d,n)$
    and on the fibers, the morphism is birational on every irreducible component.  
\end{enumerate}
\end{theorem}

\section{Wall crossing and modification in the interior of the moduli space}
\label{sec:wall-crossing}

The main purpose of this section is to prove 
Theorem~\ref{thm:locally-at-e-is-simple-blow-up}, which is the first part of Theorem \ref{thm:mainA}. We first need some preliminaries to state our result.

\begin{lemma}
\label{def:special-point-in-LM-for-lines}
Under the isomorphism $X_{\Sigma(n)}\cong\overline{\mathrm{LM}}(\mathbb{P}^d,n)$ in Theorem~\ref{thm:LM-space-for-lines-is-toric}, the identity $e$ of the dense open torus of $X_{\Sigma(n)}$ corresponds to the point in $\overline{\mathrm{LM}}(\mathbb{P}^d,n)$ parametrizing the stable hyperplane arrangement $(\mathbb{P}^d,\mathbf{t}H)$, where $H_{d+2}=\ldots=H_n=:H$ and the hyperplanes $H_1,\ldots,H_{d+1},H$ are linearly general (this configuration is pictured in Figure~\ref{fig:stable-pair-parametrized-by-e-LM-lines} for $d=2$).
\end{lemma}

\begin{proof}
By \cite[Lemma~4.2]{GR17}, we have an isomorphism
\begin{equation}
\label{eq:iso-GIT-product-proj-spaces}
(\mathbb{P}^d)^n/\!\!/_{L_{d,n}}\mathrm{SL}_{d+1}\cong(\mathbb{P}^{n-d-2})^d,
\end{equation}
where $L_{d,n}$ is an appropriate linearized line bundle. We view $(\mathbb{P}^{n-d-2})^d$ as a toric variety with the usual toric structure and identity $e$ given by $([1:\ldots:1],\ldots,[1:\ldots:1])$. By \cite[Lemma~4.2~(ii)]{GR17}, the point $([a_{d+2}^{(1)}:\ldots:a_{n}^{(1)}],\ldots,[a_{d+2}^{(d)}:\ldots:a_{n}^{(d)}])\in(\mathbb{P}^{n-d-2})^d$ parametrizes the configuration of $n$ hyperplanes in $\mathbb{P}^d$ given by
\begin{align}
\label{eq:arrangement-GIT}
\begin{split}
L_1&=V(x_0),~\ldots,~L_{d+1}=V(x_d),\\
L_i&=V(x_0+a_i^{(1)}x_1+\ldots+a_i^{(d)}x_d),~\text{for}~d+2\leq i\leq n.
\end{split}
\end{align}

On the other hand, let us view the linearized line bundle $L_{d,n}$ as a choice of weights $\mathbf{a}=(a_1,\ldots,a_n)\in\mathbb{Q}_{>0}^n$ such that $\sum_{i=1}^na_i=d+1$. More precisely, we consider
\[
\mathbf{a}:=(\underbrace{1,\ldots,1}_{d\textrm{-times}},1-\delta_1,\epsilon-\delta_2,\ldots,\epsilon-\delta_2),
\]
where $\delta_1,\delta_2$ are rational numbers satisfying the following conditions:
\begin{itemize}

\item $0<\delta_1\ll1$ and $0<\delta_2<\epsilon$;

\item $d+1-\delta_1+(n-d-1)(\epsilon-\delta_2)=d+1$;

\item $d+1-\delta_1+(n-d-1)\epsilon>d+1$.

\end{itemize}
In particular, $\mathbf{a}$ is contained in the closure of the chamber for $\mathbf{b}=(1,\ldots,1,1-\delta_1,\epsilon,\ldots,\epsilon)$. For such weight vector we have a moduli space of stable hyperplane arrangements $\overline{\mathrm{M}}_{\mathbf{b}}(\mathbb{P}^d,n)$, and by \cite[Theorem~1.5]{Ale08} we have that
\begin{equation}
\label{eq:Mb-iso-to-GIT}
\overline{\mathrm{M}}_{\mathbf{b}}(\mathbb{P}^d,n)\cong(\mathbb{P}^d)^n/\!\!/_{\mathbf{a}}\mathrm{SL}_{d+1}.
\end{equation}
We can combine \eqref{eq:iso-GIT-product-proj-spaces} and \eqref{eq:Mb-iso-to-GIT} with the reduction morphism $\overline{\mathrm{LM}}(\mathbb{P}^d,n)\rightarrow\overline{\mathrm{M}}_{\mathbf{b}}(\mathbb{P}^d,n)$ to obtain
\[
\overline{\mathrm{LM}}(\mathbb{P}^d,n)\rightarrow\overline{\mathrm{M}}_{\mathbf{b}}(\mathbb{P}^d,n)\cong(\mathbb{P}^d)^n/\!\!/_{\mathbf{a}}\mathrm{SL}_{d+1}\cong(\mathbb{P}^{n-d-2})^d.
\]
As $\overline{\mathrm{LM}}(\mathbb{P}^d,n)$ is a toric Chow quotient, the above morphism to the GIT quotient is toric by \cite[\S\,3]{KSZ91}. In particular, since it is also birational, it induces an isomorphism of the corresponding tori, preserving the identity. It follows that the stable hyperplane arrangement parametrized by $e\in\overline{\mathrm{LM}}(\mathbb{P}^d,n)$ is the one obtained by imposing $a_i^{(j)}=1$ for all $i,j$ in \eqref{eq:arrangement-GIT}, which is the claimed one.
\end{proof}

\begin{figure}
\begin{tikzpicture}[line cap=round,line join=round,>=triangle 45,x=1.0cm,y=1.0cm]
\clip(0.9,0.88) rectangle (7.1,6.4);
\draw [dashed,line width=1.0pt] (1.,1.)-- (7.,1.);
\draw [dashed,line width=1.0pt] (1.,1.)-- (4.,6.196152422706633);
\draw [dashed,line width=1.0pt] (4.,6.196152422706633)-- (7.,1.);
\draw [line width=3pt] (1.84,1.05)-- (4.68,4.95);
\draw [line width=1.0pt] (3.3,5)-- (5.94,1);
\draw [line width=1.0pt] (1.58,2.05)-- (5.48,3.56);
\draw [line width=1.0pt] (2.85,4.2)-- (6.8,1);
\draw (2.26,1.82) node[anchor=north west] {$H_4=\ldots=H_n$};
\draw (3.4,5.4) node[anchor=north west] {$H_1$};
\draw (2.9,4.6) node[anchor=north west] {$H_2$};
\draw (4.6,4.0) node[anchor=north west] {$H_3$};
\end{tikzpicture}
\caption{Stable line arrangement for the weights $\mathbf{t}=(1,1,1,\epsilon,\ldots,\epsilon)$ parametrized by the point $e\in\overline{\mathrm{LM}}(\mathbb{P}^2,n)$.}
\label{fig:stable-pair-parametrized-by-e-LM-lines}
\end{figure}
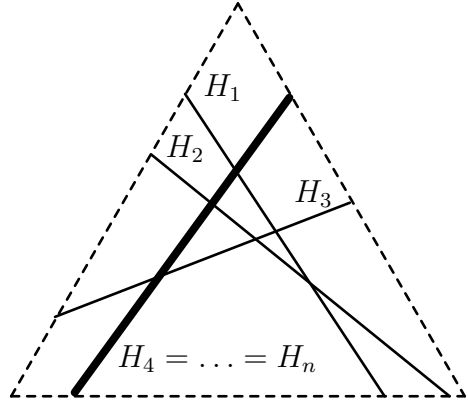

\begin{definition}
\label{def:weight_nt}
Let $0<\epsilon\ll1$, $\epsilon\in\mathbb{Q}$. We define the weight vector
\begin{align}
\label{eq:Weightd}
\mathbf{nt}=\mathbf{nt}(d,n):=\left(\underbrace{1-\epsilon,\ldots,1-\epsilon}_{(d+1)\textrm{-times}},\frac{1+\epsilon}{n-d-1},\ldots,\frac{1+\epsilon}{n-d-1}\right)\in\mathbb{Q}^n.
\end{align}
We observe that the sum of the weights in $\mathbf{nt}$ is strictly larger than $d+1$, so it defines a moduli space $\overline{\mathrm{M}}_{\mathbf{nt}}(\mathbb{P}^d,n)$. As in the case of $\mathbf{t}$, when simply writing $\mathbf{nt}$, the values of $d$ and $n$ will be clear from the context.
\end{definition}

\begin{remark}
For $\epsilon\ll1$, the sum of the weights in $\mathbf{nt}(d,n)$ is strictly larger than the sum of the weights in $\mathbf{t}(d,n)$. So, in a way, the pairs parametrized by $\overline{\mathrm{M}}_{\mathbf{nt}}(\mathbb{P}^d,n)$ are `more positive' than the ones parametrized by $\overline{\mathrm{LM}}(\mathbb{P}^d,n)$
\end{remark}

\begin{definition} 
Denote by $\mathrm{LM}(\mathbb{P}^d,n)$ the open subset of $\overline{\mathrm{LM}}(\mathbb{P}^d,n)$ parametrizing stable pairs $(\mathbb{P}^d,\mathbf{t}H)$. We point out that we are not requiring that the hyperplanes $H_1,\ldots,H_n$ are linearly general. In particular, $e\in\mathrm{LM}(\mathbb{P}^d,n)$.
\end{definition}

We are now ready to state the main result of this section.
\begin{theorem}
\label{thm:locally-at-e-is-simple-blow-up}
Let $\mathbf{nt}$ be the weight vector in Definition~\ref{def:weight_nt}. Then, the following hold:
\begin{itemize}

\item[(i)] There exists a birational morphism 
\[
f\colon\overline{\mathrm{M}}_{\mathbf{nt}}(\mathbb{P}^d,n)\rightarrow\overline{\mathrm{LM}}(\mathbb{P}^d,n).
\] 

\item[(ii)] Let us define $\mathcal{U}_{\mathbf{nt}}(\mathbb{P}^d,n):=f^{-1}(\mathrm{LM}(\mathbb{P}^d,n))$. Then,
\[
\Bl_e\mathrm{LM}(\mathbb{P}^d,n)\cong\mathcal{U}_{\mathbf{nt}}(\mathbb{P}^d,n)^\nu.
\]
\end{itemize}
\end{theorem}
The proof of Theorem~\ref{thm:locally-at-e-is-simple-blow-up} is the focus of the rest of the section. The map $f$ is obtained via Theorem \ref{thm:wall_crossing_valery}. The isomorphism between 
$\mathcal{U}_{\mathbf{nt}}(\mathbb{P}^d,n)^\nu$
and $\Bl_e\mathrm{LM}(\mathbb{P}^d,n)$ is obtained by judiciously studying the stable replacements of the pair parametrized by $e\in\overline{\mathrm{LM}}(\mathbb{P}^d,n)$ when we change the weight vector from $\mathbf{t}$ to $\mathbf{nt}$.
\subsection{Preliminary lemmas} We start by describing the structure of the chamber decomposition of $\mathcal{D}(d+1,n)$ between the weight vectors $\mathbf{t}$ and $\mathbf{nt}$.

\begin{figure}
\begin{tikzpicture}
\draw [line width=1.pt,dotted] (-5.,0.)-- (5.,0.);
\draw [line width=1.pt] (0.,3.)-- (0.,-3.);
\draw [line width=1.pt,dotted] (0.,1.)-- (5.,0.);
\draw [line width=1.pt,dash pattern=on 2pt off 2pt] (0.,3.7)-- (0.,3.);
\draw [line width=1.pt,dash pattern=on 2pt off 2pt] (0.,-3.)-- (0.,-3.7);
\draw [line width=1.pt,dash pattern=on 2pt off 2pt] (-5.666666666666667,1.)-- (-6.054956804024325,1.5824352060364884);
\draw [line width=1.pt,dash pattern=on 2pt off 2pt] (-5.666666666666667,-1.)-- (-6.054956804024325,-1.5824352060364884);
\draw [line width=1.pt,dash pattern=on 2pt off 2pt] (5.666666666666667,1.)-- (6.054956804024325,1.5824352060364884);
\draw [line width=1.pt,dash pattern=on 2pt off 2pt] (5.666666666666667,-1.)-- (6.054956804024325,-1.5824352060364884);
\draw [line width=1.pt,dash pattern=on 2pt off 2pt] (-3.,3.)-- (-2.6117098626423387,3.582435206036492);
\draw [line width=1.pt,dash pattern=on 2pt off 2pt] (-3.,-3.)-- (-2.6117098626423387,-3.582435206036492);
\draw [line width=1.pt,dash pattern=on 2pt off 2pt] (3.,-3.)-- (2.6117098626423387,-3.582435206036492);
\draw [line width=1.pt,dash pattern=on 2pt off 2pt] (3.,3.)-- (2.6117098626423387,3.582435206036492);
\draw (-5.4,0.0) node {$\mathbf{t}$};
\draw (5.5,0.0) node {$\mathbf{nt}$};
\draw (-0.3,1.3) node {$\widehat{\mathbf{w}}$};
\draw (-0.3,0.3) node {$\mathbf{w}$};
\draw (-1.8,-1.7) node {$\mathbf{a}$};
\draw (-2.5,0.4) node {$L$};
\draw (0.4,3.6) node {$W$};
\draw (0.0,-4) node {$x_{d+2}+\ldots+x_n=1$};
\draw (2.3,1.0) node {$\mathbf{h}$};
\draw [line width=1.pt] (-5.666666666666667,1.)-- (-3.,-3.);
\draw [line width=1.pt] (-5.666666666666667,-1.)-- (-3.,3.);
\draw [line width=1.pt] (3.,3.)-- (5.666666666666667,-1.);
\draw [line width=1.pt] (3.,-3.)-- (5.666666666666667,1.);
\draw [fill=black] (0.,0.) circle (3.0pt);
\draw [fill=black] (5.,0.) circle (3.0pt);
\draw [fill=black] (-5.,0.) circle (3.0pt);
\draw [fill=black] (0.,1.) circle (3.0pt);
\draw [fill=black] (2.019230769230769,0.5961538461538461) circle (3.0pt);
\draw [fill=black] (-2.,-2.) circle (3.0pt);
\end{tikzpicture}
\caption{Crossing of the wall $x_{d+2}+\ldots+x_n=1$ (pictured as the vertical line) moving from $\mathbf{t}$ to $\mathbf{nt}$ in the weight domain $\mathcal{D}(d+1,n)$. For the weight vectors $\mathbf{a},\widehat{\mathbf{w}}$, and $\mathbf{h}$, see Remark~\ref{rmk:idea-behind-the-additional-weights}.}
\label{fig:walls-chambers-and-crossing}
\end{figure}
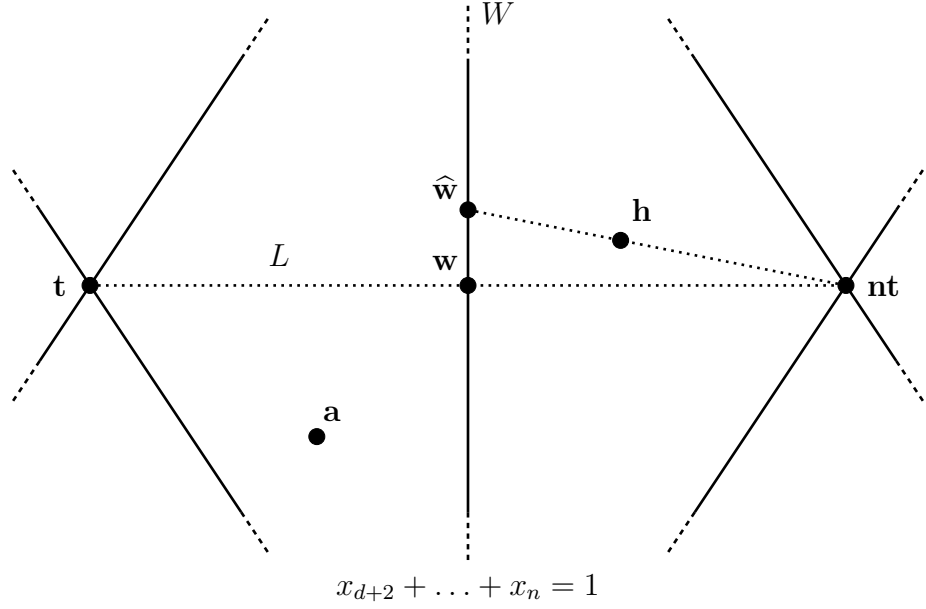

\begin{lemma}
\label{lemma:Wall_t_nt}
The following hold (see Figure~\ref{fig:walls-chambers-and-crossing}):
\begin{enumerate}

\item The weight vector $\mathbf{t}$ lies exactly on the walls $x_I=|I|$ for all $I\subseteq\{1,\ldots,d+1\}$, $|I|\geq2$.

\item The weight vector $\mathbf{nt}$ lies exactly on the walls $x_i+x_{d+2}+\ldots+x_n=2$, for all $i\in\{1,\ldots,d+1\}$.

\item The unique wall $W$ intersecting the relative interior $L^\circ$ of the segment $L$ joining $\mathbf{t}$ and $\mathbf{nt}$ is given by $x_I=1$ where $I=\{d+2,\ldots,n\}$.

\item The weight vector $\mathbf{w}=(w_1,\ldots,w_n)=W\cap L^\circ$ is given by
\[
w_i = 
\begin{cases}
1 - \epsilon + \frac{\epsilon^2}{1+\epsilon (d +2- n )}

& \text{ if }\; 1 \leq i \leq d+1
\\
\frac{1}{n-d-1} & \text{ if }\;  d+2 \leq i \leq n.
\end{cases}
\]

\end{enumerate}
\end{lemma}

\begin{proof}
(1) follows from the definition of the weight vector $\mathbf{t}$ and from the fact that $\epsilon\ll1$. To show (2), suppose that $x_I=k$ is a wall containing $\mathbf{nt}$, where $I\subseteq\{1,\ldots,n\}$, $2\leq|I|\leq n-2$, and $k\in\{1,\ldots,d\}$. Let $a=|I\cap\{1,\ldots,d+1\}|$ and $b=|I\cap\{d+2,\ldots,n\}|$, so that $0\leq a\leq d+1$ and $0\leq b\leq n-d-1$. Then, we must have that
\begin{equation}
\label{eq:proof-walls-containing-nt}
a\cdot(1-\epsilon)+b\cdot\frac{1+\epsilon}{n-d-1}=k.
\end{equation}
Suppose that $b<n-d-1$. Then, the left-hand side of \eqref{eq:proof-walls-containing-nt} would not be an integer because
\[
a\cdot(1-\epsilon)+b\cdot\frac{1+\epsilon}{n-d-1}=a+\left(-a\epsilon+b\cdot\frac{1+\epsilon}{n-d-1}\right),
\]
and the quantity in parentheses is positive and strictly less than $1$. Hence, $b=n-d-1$. So we need to find for which $0\leq a\leq d+1$, we have that $a\cdot(1-\epsilon)+(1+\epsilon)=a-a\epsilon+1+\epsilon$ is an integer. But this can only happen if $a=1$. By plugging in \eqref{eq:proof-walls-containing-nt} the values $a=1$ and $b=n-d-1$, we obtain that $k=2$. Therefore, the walls containing $\mathbf{nt}$ have the claimed form.

Let us prove (3). Suppose $W\colon x_I=k$ is a wall with the claimed property for some $I\subseteq\{1,\ldots,n\}$, $2\leq|I|\leq n-2$, and $1\leq k\leq d$. We will show that $I=\{d+2,\ldots,n\}$ and $k=1$.

Let $I_1:=I\cap\{1,\ldots,d+1\}$ and $I_2:=I\cap\{d+2,\ldots,n\}$. So, if we evaluate $x_I$ at the weight vectors $\mathbf{t}$ and $\mathbf{nt}$ we obtain
\begin{align*}
x_I(\mathbf{t})&=|I_1|+|I_2|\epsilon,\\
x_I(\mathbf{nt})&=|I_1|(1-\epsilon)+|I_2|\frac{1+\epsilon}{n-d-1}.
\end{align*}
Since, $\mathbf{t}$ and $\mathbf{nt}$ are on different sides of the wall $W$, we obtain that either $x_{I}(\mathbf{nt}) < x_{I}(\mathbf{t})$ or $x_{I}(\mathbf{t}) < x_{I}(\mathbf{nt})$. By considering the former and the latter inequalities, we find that one of the following two possibilities occurs:
\begin{align*}
 (\textrm{a}) \; |I_1|(1-\epsilon)+|I_2|\frac{1+\epsilon}{n-d-1}<|I_1|,
 && 
 (\textrm{b}) \; |I_1|+1 < |I_1|(1-\epsilon)+|I_2|\frac{1+\epsilon}{n-d-1}.
\end{align*}
Suppose that (a) holds. Then, by manipulating the inequality, we obtain that
\[
|I_2|<\epsilon((n-d-1)|I_1|-|I_2|).
\]
This forces $|I_2|=0$, which implies that $I\subseteq\{1,\ldots,d+1\}$. As $W$ intersects the relative interior of $L$, then $x_I(\mathbf{t})=|I|\neq k$. At the same time, $x_I(\mathbf{nt})=|I|(1-\epsilon)$ is arbitrarily close to $x_I(\mathbf{t})=|I|$, hence $\mathbf{t}$ and $\mathbf{nt}$ lie in the same half-space determined by $W\colon x_I=k$, which is a contradiction.

So, (b) must hold instead. After manipulating the inequality, we obtain that
\begin{equation}
\label{eq:inequality-from-b}
\epsilon\left(\frac{|I_2|}{n-d-1}-|I_1|\right)>1-\frac{|I_2|}{n-d-1}.
\end{equation}
The right-hand side of \eqref{eq:inequality-from-b} in non-negative. So,
\[
\left(\frac{|I_2|}{n-d-1}-|I_1|\right)>0\implies |I_1|<\frac{|I_2|}{n-d-1}\leq1.
\]
From this follows that $I_1=\emptyset$ because its cardinality is strictly smaller that one. So, $I=I_2$ and \eqref{eq:inequality-from-b} becomes
\begin{equation}
\label{eq:inequality-from-b-new}
\epsilon\cdot\frac{|I_2|}{n-d-1}>1-\frac{|I_2|}{n-d-1}.
\end{equation}
If the right-hand side of \eqref{eq:inequality-from-b-new} is positive, then
\[
1>\frac{|I_2|}{n-d-1}.
\]
So, $I_2$ is a proper subset of $\{d+2,\ldots,n\}$. Then, by evaluating $x_I(\mathbf{t})=|I_2|\epsilon$ and $x_I(\mathbf{nt})=(1+\epsilon)\frac{|I_2|}{n-d-1}$, we observe that $\mathbf{t}$ and $\mathbf{nt}$ cannot be on opposite sides of $W$ because, since $\epsilon\ll1$ and $\frac{|I_2|}{n-d-1} < 1$, we have that
\[
0<|I_2|\epsilon<(1+\epsilon)\frac{|I_2|}{n-d-1}<1.
\]
So, the right-hand side of \eqref{eq:inequality-from-b-new} is zero. This implies that $I_2=\{d+2,\ldots,n\}$. The last thing is to determine $k$. As $x_I(\mathbf{t})=\epsilon|I|$ and $x_I(\mathbf{nt})=1+\epsilon$, we have that $\mathbf{t}$ and $\mathbf{nt}$ are on opposite sides of the wall $W$ if and only if $k=1$. This shows part (3).

Finally, we show the formula for the weight $\mathbf{w}$ in part (4). Let $\mathbf{w}(u)=(1-u)\cdot\mathbf{t}+u\cdot\mathbf{nt}$ for $u$ in the interval $[0,1]$. We want to find $u_0$ such that $x_I(\mathbf{w}(u_0))=1$, where $I=\{d+2,\ldots,n\}$.

For $d+2\leq i\leq n$, we have that $\mathbf{w}(u)_i=(1-u)\cdot\epsilon+u\cdot\frac{1+\epsilon}{n-d-1}$. So, we have that
\begin{align*}
x_I(\mathbf{w}(u_0))=1\implies&(n-d-1)\left((1-u_0)\cdot\epsilon+u_0\cdot\frac{1+\epsilon}{n-d-1}\right)=1\\
\implies&u_0=\frac{1+\epsilon(d+1-n)}{1+\epsilon(d+2-n)}.
\end{align*}
From this we can compute the coordinates of $\mathbf{w}$ as claimed:

\begin{itemize}

\item If $1\leq i\leq d+1$, then,
\[
w_i=\mathbf{w}(u_0)_i=(1-u_0)\cdot1+u_0\cdot(1-\epsilon)=1 - \epsilon + \frac{\epsilon^2}{1+\epsilon (d +2- n )}.
\]

\item If $d+2\leq i\leq n$, then,
\[
w_i=\mathbf{w}(u_0)_i=(1-u_0)\cdot\epsilon+u_0\cdot\frac{1+\epsilon}{n-d-1}=\frac{1}{n-d-1}.
\]

\end{itemize}
\end{proof}

\begin{remark}
\label{rmk:idea-behind-the-additional-weights}
The proliferation of the weights shown in Figure~\ref{fig:walls-chambers-and-crossing} (namely, the weights $\mathbf{a},\widehat{\mathbf{w}}$, and $\mathbf{h}$) is necessary for technical purposes within the proof of Theorem~\ref{thm:locally-at-e-is-simple-blow-up}~(i). For this proof, we will use Theorem~\ref{thm:wall_crossing_valery}~(3), whose hypotheses require an certain inequality among weights. Our choice of weights is intended to satisfy such condition.
\end{remark}

\begin{lemma}
\label{lem:wall-crossing-in-the-interior-away-from-e}
Let $(\mathbb{P}^d,\mathbf{t}H)$ be a stable pair parametrized by $\overline{\mathrm{LM}}(\mathbb{P}^d,n)$. Then, $(\mathbb{P}^d,\mathbf{nt}H)$ is unstable if and only if $(\mathbb{P}^d,\mathbf{t}H)$ is the stable pair parametrized by the point $e$.
\end{lemma}

\begin{proof}
If $(\mathbb{P}^d,\mathbf{t}H)$ is parametrized by $e$, then $(\mathbb{P}^d,\mathbf{nt}H)$ is unstable because the sum of the weights of $H_{d+2},\ldots,H_n$ is strictly larger than $1$. So, let us focus on the converse.

Suppose the stable pair $(\mathbb{P}^d,\mathbf{t}H)$ is not parametrized by $e$ and let us prove that $(\mathbb{P}^d,\mathbf{nt}H)$ is stable. For this, it will be enough to check that $(\mathbb{P}^d,\mathbf{nt}H)$ is log canonical. Let $I\subseteq\{1,\ldots,n\}$ be such that $\cap_{i\in I}H_i\neq\emptyset$ and define $m_1:=|I\cap\{1,\ldots,d+1\}|$, $m_2:=|I\cap\{d+2,\ldots,n\}|$. That is, $m_1$ counts the number ``heavy'' hyperplanes with respect to $\mathbf{t}$ while $m_2$ counts the number of ``light'' ones.

By \cite[\S\,1]{Ale08}, the log canonicity of $(\mathbb{P}^d,\mathbf{t}H)$ implies that
\begin{equation}
\label{eq:inequality-assumption}
m_1+\epsilon m_2\leq\mathrm{Codim}_{\mathbb{P}^d}(\cap_{i\in I}H_i)=:c
\end{equation}
and we want to show that
\begin{equation}
\label{eq:inequality-to-verify}
(1-\epsilon)m_1+\frac{1+\epsilon}{n-d-1}m_2\leq c.
\end{equation}
We have the following cases:
\begin{itemize}

\item If $m_2=0$, then
\[
(1-\epsilon)m_1+\frac{1+\epsilon}{n-d-1}m_2=(1-\epsilon)m_1\leq m_1+\epsilon m_2\leq c.
\]

\item If $m_1=0$ and $m_2<n-d-1$, then
\[
(1-\epsilon)m_1+\frac{1+\epsilon}{n-d-1}m_2=\frac{1+\epsilon}{n-d-1}m_2\leq1\leq c,
\]
where the first inequality above holds true because $\epsilon\ll1$.

\item If $m_1=0$ and $m_2=n-d-1$, then
\[
(1-\epsilon)m_1+\frac{1+\epsilon}{n-d-1}m_2=1+\epsilon\leq c,
\]
where the last inequality holds because, otherwise, $H_{d+2}=\ldots=H_n$, which is excluded by our hypothesis.

\item We now assume $m_1,m_2>0$. By \eqref{eq:inequality-assumption} we have that $m_1\leq c-1$. Therefore,
\[
(1-\epsilon)m_1+\frac{1+\epsilon}{n-d-1}m_2\leq
(1-\epsilon)(c-1)+1+\epsilon=c+\epsilon(2-c),
\]
which is less than or equal to $c$ provided $c\geq2$. To show the latter, assume by contradiction that $c=1$. Then, the hyperplanes $H_i$, $i\in I$, must coincide. In particular, a hyperplane $H_i$ with $i\leq d+1$ would be equal to a hyperplane $H_j$ with $j\geq d+2$. This contradicts the fact that $(\mathbb{P}^d,\mathbf{t}H)$ is log canonical.
\end{itemize}
In each one of the cases above, we verified \eqref{eq:inequality-to-verify}, which concludes the proof.
\end{proof}

Now that we realized that the pair $(\mathbb{P}^d,\mathbf{t}H)$ parametrized by $e$ is unstable for the weight vector $\mathbf{nt}$, we construct some degenerate stable hyperplane arrangements with respect to $\mathbf{nt}$ that, as we will show, are naturally parametrized by the exceptional divisor of $\Bl_e\overline{\mathrm{LM}}(\mathbb{P}^d,n)$.

\begin{definition}
\label{def:description-of-the-new-fibers}
Let $Y,C_1,\ldots,C_n$ be as follows:
\begin{enumerate}

\item $Y$ is the transverse gluing of $Y_1$ and $Y_2$, where $Y_1$ is a copy of $\mathbb{P}^d$ and $Y_2$ is a copy of the blow up of $\mathbb{P}^d$ at a single point $p$. The surfaces are glued along a hyperplane $D_1\subseteq Y_1$ and the exceptional divisor $D_2\subseteq Y_2$.

\item $C_1,\ldots,C_{d+1}$ restricted to $Y_1$ give $d+1$ hyperplanes which, together with $D_1$, are in general linear position.

\item $C_1,\ldots,C_{d+1}$ restricted to $Y_2$ give the strict transform of $d+1$ distinct hyperplanes in $\mathbb{P}^d$ intersecting the exceptional divisor $D_2\cong\mathbb{P}^{d-1}$ into $d+1$ linearly general hyperplanes.

\item $C_{d+2},\ldots,C_n$ restricted to $Y_2$ give sections of the fibration $Y_2\rightarrow\mathbb{P}^{d-1}$ which are disjoint from $D_2$ and at most $n-d-2$ of these sections can coincide.

\end{enumerate}
For $d=2$, the reducible surface $Y$ together with the broken lines $C_1,\ldots,C_n$ is depicted in Figure~\ref{fig:stable-line-arrangement-after-blow-up}.
\end{definition}

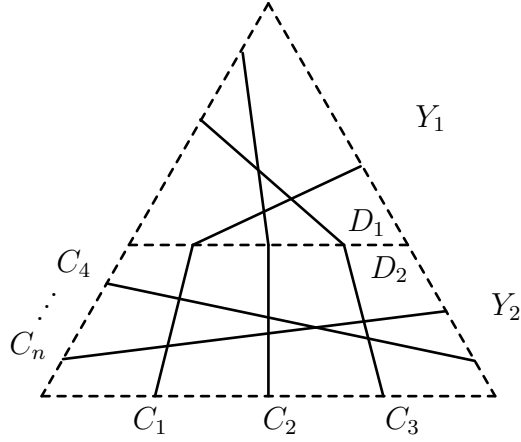
\begin{figure}
\begin{tikzpicture}[line cap=round,line join=round,>=triangle 45,x=1.0cm,y=1.0cm]
\clip(-4.3,-1) rectangle (11.48,6.3);
\draw [dashed,line width=1.0pt] (3.,5.196152422706633)-- (0.,0.);
\draw [dashed,line width=1.0pt] (3.,5.196152422706633)-- (6.,0.);
\draw [dashed,line width=1.0pt] (0.,0.)-- (6.,0.);
\draw [dashed,line width=1.0pt] (1.1547005383792517,2.)-- (4.845299461620749,2.);
\draw [line width=1.0pt] (2.,2.)-- (1.5,0.);
\draw [line width=1.0pt] (3.,2.)-- (3.,0.);
\draw [line width=1.0pt] (4.,2.)-- (4.52,0.);
\draw [line width=1.0pt] (0.8608587988004848,1.491051177665153)-- (5.730814157129579,0.46624355652982175);
\draw [line width=1.0pt] (0.2828460969082651,0.4899038105676654)-- (5.3507437415779595,1.1245448267190437);
\draw [line width=1.0pt] (2.,2.)-- (4.02 + 0.2,2.84 + 0.2);
\draw [line width=1.0pt] (4.,2.)-- (2.32 - 0.22 ,3.46 + 0.2);
\draw [line width=1.0pt] (3.,2.)-- (2.76 - 0.1 , 4.04 +0.5);
\draw (1.06,0) node[anchor=north west] {$C_1$};
\draw (2.78,0) node[anchor=north west] {$C_2$};
\draw (4.4,0) node[anchor=north west] {$C_3$};
\draw (0.2-0.15,2.22-0.15) node[anchor=north west] {$C_4$};
\draw (-0.42-0.15,1.14-0.15) node[anchor=north west] {$C_n$};
\draw (-0.06-0.15,1.76-0.15) node[anchor=north west] {$ \rotatebox[origin=c]{60}{\dots}$};
\draw (4.3,2.3) node {$D_1$};
\draw (4.2,2) node[anchor=north west] {$D_2$};
\draw (4.8,4) node[anchor=north west] {$Y_1$};
\draw (5.8,1.5) node[anchor=north west] {$Y_2$};
\end{tikzpicture}
\caption{The reducible surface $Y$ with the broken lines $C_1,\ldots,C_n$ described in Definition~\ref{def:description-of-the-new-fibers}.}
\label{fig:stable-line-arrangement-after-blow-up}
\end{figure}

\begin{lemma}
\label{lem:stability-pair-for-weights-nt}
Let $Y,C_1,\ldots,C_n$ as in Definition~\ref{def:description-of-the-new-fibers}. Then, the pair $(Y,\mathbf{nt}C)$ is stable.
\end{lemma}

\begin{proof}
It follows from the definition of $\mathbf{nt}$ and Definition~\ref{def:description-of-the-new-fibers} that $K_Y+\mathbf{nt}C$ is $\mathbb{Q}$-Cartier and that the pairs $(Y_1,D_1+\mathbf{nt}C|_{Y_1})$ and $(Y_2,D_2+\mathbf{nt}C|_{Y_2})$ are log canonical. So, we only have to check that the divisors $D_1+K_{Y_1}+\mathbf{nt}C|_{Y_1}$ and $D_2+K_{Y_2}+\mathbf{nt}C|_{Y_2}$ are ample.

If $H\subseteq Y_1$ denotes a generic hyperplane, then
\[
D_1+K_{Y_1}+\mathbf{nt}C|_{Y_1}\sim_{\mathbb{Q}}H-(d+1)H+(d+1)(1-\epsilon)H=(1-(d+1)\epsilon)H,
\]
which is ample because $\epsilon\ll1$. So, let us focus on the second component $Y_2$. Let $e\subseteq Y_2$ be any line contained in the exceptional divisor $D_2$. Let $f\subseteq Y_2$ denote the strict transform of a line in $\mathbb{P}^d$ passing through $p$. Finally, let $s\subseteq Y_2$ be the strict transform of a line in $\mathbb{P}^d$ disjoint from $p$. To prove that $D_2+K_{Y_2}+\mathbf{nt}C|_{Y_2}$ is ample, it is sufficient to prove that it intersects positively with the curves $e,f,s$ (this is a combination of Kleiman's ampleness criterion \cite[Theorem~1.18]{KM98} with the fact that in a toric variety the cone of pseudoeffective curves is generated by the boundary curves \cite[Theorem~6.3.20~(b)]{CLS11}). Let $H\subseteq Y_2$ (resp. $R\subseteq Y_2$) be the strict transform of a hyperplane in $\mathbb{P}^d$ not passing through $p$ (resp. passing through $p$). We have that $K_{Y_2}\sim-(d+1)H+(d-1)D_2$ by \cite[Page~187]{GH94} and $R\sim H-D_2$ by \cite[Page~605]{GH94}. 
Moreover, $C|_{Y_2}$ is the sum of the divisors 
$(C_1 + \ldots + C_{d+1})|_{Y_2}$ which is linearly equivalent to $(d+1)R$
and 
$(C_{d+2} + \ldots+ C_n)|_{Y_2}$ which is linearly equivalent to $(n-d-1)H$.
Therefore,
\begin{align*}
D_2+K_{Y_2}+C|_{Y_2} \sim_{\mathbb{Q}}D_2+{}&(-(d+1)H+(d-1)D_2)\\
+{}&\left((d+1)(1-\epsilon)(H-D_2)+(n-d-1)\frac{1+\epsilon}{n-d-1}H\right)\\
={}&(1-\epsilon d)H+(\epsilon(1+d)-1)D_2.
\end{align*}
From the geometry of $Y_2$ we can argue that
\begin{align*}
H\cdot e & =0 &&  H\cdot f =1 && H\cdot s =1
\\
D_2\cdot e &= -1 &&  D_2\cdot f =1 && D_2\cdot s =0,
\end{align*}
from which we obtain that
\begin{align*}
\left( D_2+K_{Y_2}+C|_{Y_2}\right)\cdot e &= 1-\epsilon(1+d)
\\
\left( D_2+K_{Y_2}+C|_{Y_2} \right)\cdot f &= \epsilon
\\
\left( D_2+K_{Y_2}+C|_{Y_2} \right)\cdot s &= 1-\epsilon d,
\end{align*}
which are all positive for $0<\epsilon\ll1$.
\end{proof}

\subsection{Proof of Theorem~\ref{thm:locally-at-e-is-simple-blow-up}}
We are now ready to assemble all our work together.

\begin{proof}[Proof of part~(i)]
We show that  there is a birational morphism $\overline{\mathrm{M}}_{\mathbf{nt}}(\mathbb{P}^d,n)\rightarrow\overline{\mathrm{LM}}(\mathbb{P}^d,n)$.
The strategy is defining auxiliary weight vectors $\mathbf{a},\widehat{\mathbf{w}}$, and $\mathbf{h}$ for which we can use Theorem~\ref{thm:wall_crossing_valery}. Let us define
\[
\mathbf{a}:=\left(\underbrace{1-\frac{1}{d+1}+\widehat{\epsilon},\ldots,1-\frac{1}{d+1}+\widehat{\epsilon}}_{(d+1)\textrm{-times}},\epsilon,\ldots,\epsilon\right)\in\mathbb{Q}^n,
\]
where
\[
\frac{1}{d+1}-\frac{(n-d-1)\epsilon}{d+1}<\widehat{\epsilon}<\frac{1}{d+1}.
\]
With this definition, $\sum_ia_i>d+1$, so that we can consider the moduli space of stable hyperplane arrangements $\overline{\mathrm{M}}_{\mathbf{a}}(\mathbb{P}^d,n)$. We observe that $\mathbf{a}$ is not contained in any wall of the chamber decomposition of $\mathcal{D}(d+1,n)$ and that it can be chosen to be arbitrarily close to $\mathbf{t}$, 
so $\mathbf{t}$ is in the closure of the chamber containing $\mathbf{a}$. 
As we also have that $\mathbf{a}\leq\mathbf{t}$, then we can conclude by Theorem~\ref{thm:wall_crossing_valery}~(3) that $\overline{\mathrm{LM}}(\mathbb{P}^d,n)=\overline{\mathrm{M}}_{\mathbf{a}}(\mathbb{P}^d,n)$.

We now define
\[
\widehat{\mathbf{w}}=\left(1-\frac{1}{d+1}+\widehat{\epsilon},\ldots,1-\frac{1}{d+1}+\widehat{\epsilon},\frac{1}{n-d-1},\ldots,\frac{1}{n-d-1}\right),
\]
which lies on the wall $x_{d+2}+\ldots+x_n=1$. As there is no wall between $\mathbf{a}$ and $\widehat{\mathbf{w}}$ (this can be shown using ideas analogous to the ones used in Lemma~\ref{lemma:Wall_t_nt}~(3)), $\widehat{\mathbf{w}}$ is in the closure of the chamber containing $\mathbf{a}$. Moreover, $\mathbf{a}\leq\widehat{\mathbf{w}}$. So, by Theorem~\ref{thm:wall_crossing_valery}~(3) we have that $\overline{\mathrm{M}}_{\mathbf{a}}(\mathbb{P}^d,n)=\overline{\mathrm{M}}_{\widehat{\mathbf{w}}}(\mathbb{P}^d,n)$.

We notice that the weight vectors $\widehat{\mathbf{w}}$ and $\mathbf{w}$ can be chosen to be arbitrarily close for $\epsilon\ll1$, so there is no wall between them. Hence $\overline{\mathrm{M}}_{\widehat{\mathbf{w}}}(\mathbb{P}^d,n)=\overline{\mathrm{M}}_{\mathbf{w}}(\mathbb{P}^d,n)$.

Let $\mathbf{h}$ be a weight vector in the relative interior of the segment joining $\widehat{\mathbf{w}}$ and $\mathbf{nt}$. Since $\widehat{\mathbf{w}}\leq\mathbf{nt}$ (for this, we need that $\widehat{\epsilon}<\frac{1}{d+1}$), we have that $\mathbf{h}\leq\mathbf{nt}$. Since $\widehat{\mathbf{w}}$ and $\mathbf{w}$ are arbitrarily close, there is no wall between $\widehat{\mathbf{w}}$ and $\mathbf{nt}$. Therefore, $\mathbf{nt}$ is in the closure of the chamber containing $\mathbf{h}$. It follows by Theorem~\ref{thm:wall_crossing_valery}~(3) that $\overline{\mathrm{M}}_{\mathbf{nt}}(\mathbb{P}^d,n)=\overline{\mathrm{M}}_{\mathbf{h}}(\mathbb{P}^d,n)$.

Finally, $\mathbf{w}$ is in the closure of the chamber containing $\mathbf{h}$, which by Theorem~\ref{thm:wall_crossing_valery}~(2) guarantees the existence of a birational morphism $\overline{\mathrm{M}}_{\mathbf{h}}(\mathbb{P}^d,n)\rightarrow\overline{\mathrm{M}}_{\mathbf{w}}(\mathbb{P}^d,n)$.

Summarizing the considerations above, we have that
\[
\overline{\mathrm{M}}_{\mathbf{nt}}(\mathbb{P}^d,n)=\overline{\mathrm{M}}_{\mathbf{h}}(\mathbb{P}^d,n)\rightarrow\overline{\mathrm{M}}_{\mathbf{w}}(\mathbb{P}^d,n)=\overline{\mathrm{M}}_{\widehat{\mathbf{w}}}(\mathbb{P}^d,n)=\overline{\mathrm{M}}_{\mathbf{a}}(\mathbb{P}^d,n)=\overline{\mathrm{LM}}(\mathbb{P}^d,n),
\]
which gives the claimed birational morphism.
\end{proof}

\begin{proof}[Proof of part (ii)]
Consider the birational morphism $\overline{\mathrm{M}}_{\mathbf{nt}}(\mathbb{P}^d,n)\rightarrow\overline{\mathrm{LM}}(\mathbb{P}^d,n)$ from part~(i). This is an isomorphism over $\mathrm{LM}(\mathbb{P}^d,n)\setminus\{e\}$ by Lemma~\ref{lem:wall-crossing-in-the-interior-away-from-e}. Then, denote by
\[
\varphi\colon\Bl_e\mathrm{LM}(\mathbb{P}^d,n)\dashrightarrow\mathcal{U}_{\mathbf{nt}}(\mathbb{P}^d,n)
\] 
the isomorphism away from the exceptional divisor $\mathbb{E}_n\subseteq\Bl_e\mathrm{LM}(\mathbb{P}^d,n)$. We have that $\varphi$ extends to a morphism $\overline{\varphi}\colon\Bl_e\mathrm{LM}(\mathbb{P}^d,n)\rightarrow\mathcal{U}_{\mathbf{nt}}(\mathbb{P}^d,n)$.

To prove this, we use \cite[Lemma~5.3]{GPSZ24} (which is a slight generalization of \cite[Lemma~3.18]{AET23} and \cite[Theorem~7.3]{GG14}). Consider the rational map $\Bl_e\overline{\mathrm{LM}}(\mathbb{P}^d,n)\dashrightarrow\overline{\mathrm{M}}_{\mathbf{nt}}(\mathbb{P}^d,n)$. Let $U\subseteq\Bl_e\overline{\mathrm{LM}}(\mathbb{P}^d,n)$ be the open subset given by $\Bl_e\mathrm{LM}(\mathbb{P}^d,n)\setminus\mathbb{E}_n$. Let $x\in\mathbb{E}_n$ be arbitrary. Let $R$ be a DVR with maximal ideal $\mathfrak{m}$ and let $K$ be its field of fractions. Let $g\colon\mathrm{Spec}(K)\rightarrow U$ which extends to $\overline{g}\colon\mathrm{Spec}(R)\rightarrow\Bl_e\overline{\mathrm{LM}}(\mathbb{P}^d,n)$ with $\overline{g}(\mathfrak{m})=x$. Let $V\subseteq\Bl_e\overline{\mathrm{LM}}(\mathbb{P}^d,n)$ be the open subset given by $\Bl_e\mathrm{LM}(\mathbb{P}^d,n)$, which contains $U$. Let $f:=\varphi\circ g$ and let $\overline{f}$ be its unique extension to an $R$-point of $\overline{\mathrm{M}}_{\mathbf{nt}}(\mathbb{P}^d,n)$, which by continuity has to lie in $\mathcal{U}_{\mathbf{nt}}(\mathbb{P}^d,n)$. Let us show that $y:=\overline{f}(\mathfrak{m})$ only depends on $x$ and not on the choice of $g$. This implies that $\varphi$ extends to a birational morphism $\overline{\varphi}\colon V=\Bl_e\mathrm{LM}(\mathbb{P}^d,n)\rightarrow\mathcal{U}_{\mathbf{nt}}(\mathbb{P}^d,n)$.

Consider the family $(\mathcal{X},\mathcal{B}_1,\ldots,\mathcal{B}_n)\rightarrow\overline{\mathrm{LM}}(\mathbb{P}^d,n)$ from Theorem~\ref{thm:family-over-moduli-weighted-hyperplane-arrangements}. Denote by $\pi$ the blow up morphism $\Bl_e\overline{\mathrm{LM}}(\mathbb{P}^d,n)\rightarrow\overline{\mathrm{LM}}(\mathbb{P}^d,n)$. Then consider the one-parameter family
\begin{equation}
\label{eq:one-parameter-family-to-modify}
(\mathcal{X}',\mathcal{C}_1,\ldots,\mathcal{C}_n):=(\pi\circ \overline{g})^*(\mathcal{X},\mathcal{B}_1,\ldots,\mathcal{B}_n)\rightarrow\mathrm{Spec}(R),
\end{equation}
whose central fiber $(\mathbb{P}^d,\mathbf{t}H)$ is the one described in Lemma~\ref{def:special-point-in-LM-for-lines} (for $d=2$, this is depicted in Figure~\ref{fig:stable-pair-parametrized-by-e-LM-lines}). The point $y$ parametrizes the stable replacement of the central fiber of the one-parameter family in \eqref{eq:one-parameter-family-to-modify} with respect to the weight vector $\mathbf{nt}$.

To compute this stable replacement, let $Z_1\subseteq\mathcal{X}'$ be the closed subset of the central fiber where the hyperplanes $H_{d+2},\ldots,H_n$ coincide. Let $\mathcal{X}_1':=\Bl_{Z_1}\mathcal{X}'$ and denote by $E_1$ its exceptional divisor, which is isomorphic to the blow up of $\mathbb{P}^d$ at one point. For $i=1,\ldots,n$, let $\widehat{\mathcal{C}}_i$ be the strict transform of $\mathcal{C}_i$ with respect to the blow up $\mathcal{X}_1'\rightarrow\mathcal{X}'$. The restrictions $\widehat{\mathcal{C}}_{d+2}|_{E_1},\ldots,\widehat{\mathcal{C}}_n|_{E_1}$ may still be all equal or not. If not, define $Z_2$ to be the hyperplane where these strict transforms coincide and we iterate this construction. This gives rise to a sequence of blow ups $\mathcal{X}_s'\rightarrow\ldots\rightarrow\mathcal{X}_1'\rightarrow\mathcal{X}'$ where the strict transforms of $\mathcal{C}_{d+2},\ldots,\mathcal{C}_n$ on $E_s$ have at least two irreducible components for the first time. This procedure terminates because after every blow up the multiplicity of intersection of $\widehat{\mathcal{C}}_{d+2},\ldots\widehat{\mathcal{C}}_{n}$ along the central fiber decreases. The central fiber of $\mathcal{X}_s'\rightarrow\mathrm{Spec}(R)$ for $d=2$ is depicted in Figure~\ref{fig:stable-replacement-before-contraction}.

If $s=1$, then this is the stable replacement by Lemma~\ref{lem:stability-pair-for-weights-nt}. If $s\geq2$, then, to obtain the stable replacement, we need to contract the components $E_1,\ldots,E_{s-1}$ of the central fiber. The reason is the following. Let $E_h$ be one of these components, which is isomorphic to the blow up of $\mathbb{P}^d$ at one point $p$. Let $f\subseteq E_h$ be the strict transform of a line passing through $p$ and $D_h\subseteq E_h$ the conductor divisor. The conductor $D_h$ has two connected components given by the exceptional divisor $E\subseteq E_h$ and the strict transform $H$ of a hyperplane not passing through $p$. Then, $f$ intersects the log canonical divisor giving $0$:
\[
(K_{E_h}+D_h+(1-\epsilon)(\widehat{\mathcal{C}}_1|_{E_h}+\ldots+\widehat{\mathcal{C}}_{k+1}|_{E_h}))\cdot f=-2+2+0=0.
\]
To compute $K_{E_h}\cdot f=-2$ we used that $K_{E_h}=(-d-1)H+(d-1)E$. So, $K_{E_h}\cdot f=(-d-1)+(d-1)=-2$. Therefore, we can conclude that there exists a birational morphism $\mathcal{X}_k'\rightarrow\mathcal{Y}$ contracting $E_1,\ldots,E_{s-1}$ along their $\mathbb{P}^1$-fibrations. Still denote by $\widehat{\mathcal{C}}_1,\ldots,\widehat{\mathcal{C}}_n$ the strict transforms of $\mathcal{C}_1,\ldots,\mathcal{C}_n$ on $\mathcal{Y}$. The central fiber of $(\mathcal{Y},\widehat{\mathcal{C}}_1,\ldots,\widehat{\mathcal{C}}_n)\rightarrow\mathrm{Spec}(R)$ is then as described in Definition~\ref{def:description-of-the-new-fibers} (see Figure~\ref{fig:stable-line-arrangement-after-blow-up} for the case $d=2$), which is stable for the weight vector $\mathbf{nt}$ by Lemma~\ref{lem:stability-pair-for-weights-nt}.

We now need to show that the central fiber of $(\mathcal{Y},\widehat{\mathcal{C}}_1,\ldots,\widehat{\mathcal{C}}_n)\rightarrow\mathrm{Spec}(R)$, hence the point $y$, is independent of the choice of $g$ and only depends on $x$. The central fiber of $(\mathcal{Y},\widehat{\mathcal{C}}_1,\ldots,\widehat{\mathcal{C}}_n)\rightarrow\mathrm{Spec}(R)$ consists of $X_0'\cong\mathbb{P}^d$, which is the strict transform of the central fiber $X_0\subseteq\mathcal{X}$, glued along a hyperplane $H_0$ with the exceptional divisor $D_s\subseteq E_s$. The isomorphism type of the central fiber of $\mathcal{Y}\rightarrow\mathrm{Spec}(R)$ only depends on the pair
\begin{equation}
\label{eq:tail-of-the-stable-replacement}
\left(E_s,D_s+(1-\epsilon)(\widehat{\mathcal{C}}_1|_{E_s}+\ldots+\widehat{\mathcal{C}}_{d+1}|_{E_s})+\frac{1+\epsilon}{n-d-1}\left(\widehat{\mathcal{C}}_{d+2}|_{E_s}+\ldots+\widehat{\mathcal{C}}_n|_{E_s}\right)\right),
\end{equation}
because the component $(X_0',H_0+(1-\epsilon)(\widehat{\mathcal{C}}_1|_{X_0'}+\ldots+\widehat{\mathcal{C}}_{d+1}|_{X_0'}))$ is unique up to automorphisms, as the divisor is supported on $d+2$ linearly general hyperplanes in $\mathbb{P}^d$. The stable pair in \eqref{eq:tail-of-the-stable-replacement} depends on how the sections $\mathcal{C}_{d+2},\ldots,\mathcal{C}_n$ collide in the central fiber of $\mathcal{Y}\rightarrow\mathrm{Spec}(R)$. But this is prescribed by $x\in\mathbb{E}_n$ and not by $g$. We now elaborate on this.

The limiting behavior can be locally described as follows. Let $\Delta=\mathrm{Spec}(\mathbb{C}[[t]])$. Consider the following one-parameter family of hyperplanes in affine space:
\[
L(t)=V(a_0(t)+a_1(t)x_1+\ldots+a_d(t)x_d)\subseteq\mathbb{A}_{(x_1,\ldots,x_d)}^d\times\Delta,
\]
where $a_0(t),\ldots,a_d(t)\in\mathbb{C}[[t]]$, $a_1(0)\neq0$, and $a_0(0),a_2(0),\ldots,a_d(0)=0$. Therefore, for $t\to0$, the limiting hyperplane is $Y:=V(x_1,t)\subseteq\mathbb{A}^d\times\Delta$. Let us consider the blow up
\[
\Bl_Y(\mathbb{A}^d\times\Delta).
\]
Denote by $\widehat{L}(t)$ the strict transform of $L(t)$ and by $E$ the exceptional divisor of the blow up. Let us compute $\widehat{L}(0)\subseteq E$.

First, we give equations for $\Bl_Y(\mathbb{A}^d\times\Delta)$. As $Y=V(x_1,t)\subseteq\mathbb{A}^d\times\Delta$, we have that $\Bl_Y(\mathbb{A}^d\times\Delta)$ can be viewed as the closed subvariety of $(\mathbb{A}^d\times\Delta)\times\mathbb{P}_{[s_0:s_1]}^1$ given by $x_1s_1-ts_0=0$. The exceptional divisor $E$ is described by $x_1=t=0$, so $((x_2,\ldots,x_d),[s_0:s_1])\in\mathbb{A}^{d-1}\times\mathbb{P}^1$ are free to vary, revealing the $\mathbb{P}^1$-bundle structure of $E$. More explicitly, we can write that
\[
E=\{((0,x_2,\ldots,x_d),0,[s_0:s_1])\in(\mathbb{A}^d\times\Delta)\times\mathbb{P}^1\mid (x_2,\ldots,x_d)\in\mathbb{A}^{d-1},~[s_0:s_1]\in\mathbb{P}^1\}.
\]

Now, to compute $\widehat{L}(0)\subseteq E$ it suffices to look in the affine patch $s_1\neq0$, where we set $\sigma_0=s_0/s_1$. By substituting $x_1=t\sigma_0$ in the defining equation of $L(t)$ we obtain the pullback of $L(t)$ with respect to the blow up restricted to $s_1 \neq 0$, that is 
\[
 a_0(t)+a_1(t)t\sigma_0+a_2(t)x_2+\ldots+a_d(t)x_d=0.
\]
From the above expression, we obtain the strict transform and the exceptional divisor as follows
\[
t\left(\frac{a_0(t)}{t}+a_1(t)\sigma_0+\frac{a_2(t)}{t}x_2+\ldots+\frac{a_d(t)}{t}x_d\right)=0.
\]
The vanishing of $t$ describes the exceptional divisor. Hence, the strict transform $\widehat{L}(0)$ is given by
\[
\lim_{t\to0}\left(\frac{a_0(t)}{t}+a_1(t)\sigma_0+\frac{a_2(t)}{t}x_2+\ldots+\frac{a_d(t)}{t}x_d\right)=0.
\]
By algebraic manipulations of the above limit, we obtain that
\[
\sigma_0=-\frac{1}{a_1(0)}\lim_{t\to0}\left(\frac{a_0(t)}{t}+\frac{a_2(t)}{t}x_2+\ldots+\frac{a_d(t)}{t}x_d\right)=-\frac{1}{a_1(0)}K_0.
\]
The limit $K_0$ exists because, recall, $a_0(t),a_2(t),\ldots,a_d(t)\in\mathbb{C}[[t]]$ vanish at $t=0$. Therefore,
\[
\widehat{L}(0)=\left\{\left((0,x_2,\ldots,x_d),0,\left[K_0:-a_1(0)\right]\right)\in(\mathbb{A}^d\times\Delta)\times\mathbb{P}^1\,\bigg|\,(x_2,\ldots,x_d)\in\mathbb{A}^{d-1}\right\}\subseteq E.
\]
So, $\widehat{L}(0)$ is a section of the $\mathbb{P}^1$-bundle $E$, and we can see it only depends on $a_1(0)$ and first order information about $a_0(t),a_2(t),\ldots,a_d(t)$. These only depend on $x\in\mathbb{E}_n$ and not on the choice of one-parameter family limiting to it.

Finally, let us prove that $\overline{\varphi}$ is finite. The morphism $\overline{\varphi}$ is injective away from $\mathbb{E}_n$, so we only need to prove finiteness over the divisor $\mathbb{E}_n$. The locus in $\overline{\mathrm{M}}_{\mathbf{nd}}(\mathbb{P}^2,n)$ parametrizing the isomorphism classes of stable pairs as in Figure~\ref{fig:stable-line-arrangement-after-blow-up} gives a divisor. This is because, after contracting the $\mathbb{P}^d$-component of the degeneration, we have a hyperplane arrangement such that $C_1\cap\ldots\cap C_{d+1}$ equals one point, which is a codimension one condition. So, $\overline{\varphi}$ is generically finite on $\mathbb{E}_n$ and we only need to show that $\overline{\varphi}$ does not contract any curve. But this is true because $e$ is a smooth point of $\overline{\mathrm{LM}}(\mathbb{P}^d,n)$, hence $\mathbb{E}_n$ is isomorphic to $\mathbb{P}^{N-1}$, where $N:=d(n-d-2)=\dim\overline{\mathrm{LM}}(\mathbb{P}^d,n)$. So, if a curve is contracted, the whole $\mathbb{E}_n$ is contracted, which is impossible as $\overline{\varphi}|_{\mathbb{E}_n}$ is generically finite. This proves quasi-finiteness of $\overline{\varphi}$. Next, we need to show that $\overline{\varphi}$ is proper. Consider the following composition:
\[
\Bl_e\mathrm{LM}(\mathbb{P}^d,n)\xrightarrow{\overline{\varphi}}\mathcal{U}_{\mathbf{nt}}(\mathbb{P}^d,n)\rightarrow\mathrm{LM}(\mathbb{P}^d,n).
\]
The above composition is proper because it is the base change of the proper morphism $\Bl_e\overline{\mathrm{LM}}(\mathbb{P}^d,n)\rightarrow\overline{\mathrm{LM}}(\mathbb{P}^d,n)$. Also, the morphism $\mathcal{U}_{\mathbf{nt}}(\mathbb{P}^d,n)\rightarrow\mathrm{LM}(\mathbb{P}^d,n)$ is proper because it is the base change of the proper morphism $\overline{\mathrm{M}}_{\mathbf{nt}}(\mathbb{P}^d,n)\rightarrow\overline{\mathrm{LM}}(\mathbb{P}^d,n)$. This implies properness of $\overline{\varphi}$. So, by the Zariski Main Theorem, we can conclude that the lift of $\overline{\varphi}$ to the normalization $\widetilde{\varphi}\colon\Bl_e\mathrm{LM}(\mathbb{P}^d,n)\rightarrow\mathcal{U}_{\mathbf{nt}}(\mathbb{P}^d,n)^\nu$ is an isomorphism.
\end{proof}

\begin{figure}
\begin{tikzpicture}[line cap=round,line join=round,>=triangle 45,x=1.0cm,y=1.0cm]
\draw (11.016665777829033,-2.94462641198056-0.1) node[anchor=north west] {$C_1$};
\draw (13.633154130355981,-2.927960881072745-0.1) node[anchor=north west] {$C_2$};
\draw (16.199645890159484,-2.91129535016493-0.1) node[anchor=north west] {$C_3$};
\draw (9.350112687047538-0.2,-0.5614554921630157-0.2) node[anchor=north west] {$C_4$};
\draw (8.58349826528805-0.2,-1.9446945575116603-0.2) node[anchor=north west] {$C_n$};
\draw (9.050133130706868-0.2,-1.1114180121209105-0.2) node[anchor=north west] {$ \rotatebox[origin=c]{60}{\dots}$};
\draw [dashed,line width=1.0pt] (9.,-3.)-- (19.,-3.);
\draw [dashed,line width=1.0pt] (14.,5.660254037844387)-- (11.636751345948136,1.5669872981077881);
\draw [dashed,line width=1.0pt] (14.,5.660254037844387)-- (16.363248654051873,1.5669872981077841);
\draw [dashed,line width=1.0pt] (9.,-3.)-- (11.270725942163699,0.933012701892227);
\draw [dashed,line width=1.0pt] (19.,-3.)-- (16.729274057836314,0.9330127018922194);
\draw [line width=1.0pt] (13.232050807568879,3.)-- (12.818375672974065,1.5669872981077873);
\draw [line width=1.0pt] (14.,3.)-- (14.,1.5669872981077864);
\draw [line width=1.0pt] (14.767949192431125,3.)-- (15.181624327025935,1.5669872981077853);
\draw [line width=1.0pt] (12.635362971081847,0.9330127018922251)-- (11.5,-3.);
\draw [line width=1.0pt] (14.,0.9330127018922234)-- (14.,-3.);
\draw [ line width=1.0pt] (15.364637028918153,0.9330127018922215)-- (16.5,-3.);
\draw [line width=1.0pt] (13.232050807568879,3.)-- (14.766410232087397 + 0.18, 3.8715757293157784+ 0.18);
\draw [line width=1.0pt] (14.,3.) -- 
(13.79980943943413  + 0.0, 4.721517805614344 + 0.65);
\draw [line width=1.0pt] 
(14.767949192431125,3.)-- (13.483164352185646 -0.15,4.338210594734599 + 0.2);
\draw [dashed,line width=1.0pt] (11.020725942163693,0.5)-- (16.979274057836314,0.5);
\draw [dashed,line width=1.0pt] (11.886751345948133,2.)-- (16.113248654051876,2.);
\draw [dashed,line width=1.0pt] (12.464101615137759,3.)-- (15.53589838486225,3.);
\draw [dashed,line width=1.0pt] (10.154700538379252,-1.)-- (17.84529946162075,-1.);
\draw [line width=1.0pt] (9.30218169330195,-2.476605954083826)-- (18.068394916438923,-1.3864126626827598);
\draw [line width=1.0pt] (9.995633221760974,-1.2755126742065057)-- (18.494241876281613,-2.124001233379047);
\draw (14.616420453917064,5.421470103742568) node[anchor=north west] {$\mathbb{P}^2$};
\draw (17.516222831876867,0.255155522319919) node[anchor=north west] {$E_{s-1}$};
\draw (15.899666333818816,3.071630245740653) node[anchor=north west] {$E_1$};
\draw (18.699475526331728,-1.4447286302772104) node[anchor=north west] {$E_s$};
\end{tikzpicture}
\caption{For $d=2$, semistable replacement in the proof of Theorem~\ref{thm:locally-at-e-is-simple-blow-up} before the contraction. The components $E_1,\ldots,E_s$ are isomorphic to $\mathbb{F}_1$.}
\label{fig:stable-replacement-before-contraction}
\end{figure}
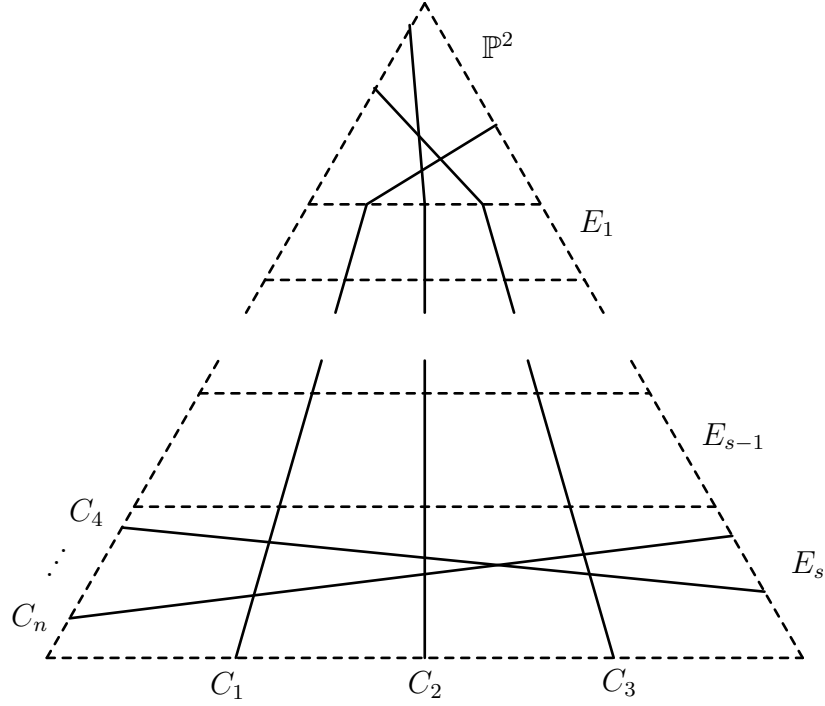


\section{The case of lines in the projective plane}
\label{sec:iso-in-case-of-lines}

In this section, we prove the Theorem~\ref{thm:geometric-meaning-blow-up-identity-dim2}, which covers the case of lines in $\mathbb{P}^2$ claimed in Theorem~\ref{thm:mainA}.

\begin{theorem}
\label{thm:geometric-meaning-blow-up-identity-dim2}
Let $\mathbf{nt}$ be as in Definition~\ref{def:weight_nt}. Then, $\Bl_e\overline{\mathrm{LM}}(\mathbb{P}^2,n)\cong\overline{\mathrm{M}}_{\mathbf{nt}}(\mathbb{P}^2,n)^\nu$.
\end{theorem}

The proof amounts to show that every stable pair parametrized by 
$\overline{\mathrm{LM}}(\mathbb{P}^2,n)\setminus\{e\}$ corresponds to one parametrized by $\overline{\mathrm{M}}_{\mathbf{nt}}(\mathbb{P}^2,n)$. However, there is a technical aspect that makes the writing a bit more nuanced. The log canonical divisor is $\mathbb{Q}$-Cartier for the pairs parametrized by $\overline{\mathrm{LM}}(\mathbb{P}^2,n)$, but such feature is lost in the wall crossing towards  $\overline{\mathrm{M}}_{\mathbf{nt}}(\mathbb{P}^2,n)$. The next lemma makes this precise.

\begin{lemma}
\label{lem:Q-Cartier-modification}
Let $(X,\mathbf{t}C)$ be a stable pair parametrized by $\overline{\mathrm{LM}}(\mathbb{P}^2,n)$. Let $X=\cup_j X_j$ be the decomposition into irreducible components. Then, $K_X+\mathbf{nt}C$ is not $\mathbb{Q}$-Cartier if and only if there exists an irreducible component $X_k$ isomorphic to $\mathbb{P}^1\times\mathbb{P}^1$ such that exactly one of the intersections $X_k\cap C_1,X_k\cap C_{2},X_k\cap C_{3}$ is a point.
\end{lemma}

\begin{proof}
This follows from \cite[Theorem~5.7.2]{Ale15}. The reason why at most one point is blown up boils down to the combinatorics of mixed subdivisions of $(n-3)\Delta_2$. If a subpolytope of $(n-3)\Delta_2$ corresponding to $\mathbb{P}^1\times\mathbb{P}^1$ intersects the boundary of $(n-3)\Delta_2$ in an isolated point, then this can happen only once (see \cite[Figure~3]{San05}).
\end{proof}

\begin{definition}
\label{def:Q-Cartier-modification}
Let $(X,\mathbf{t}C)$ be a stable pair parametrized by a point in $
\overline{\mathrm{LM}}(\mathbb{P}^2,n)$. We define 
$(\widetilde{X},\mathbf{nt}\widetilde{C})$ as follows
\begin{itemize}
    \item If $K_X+\mathbf{nt}C$ is not $\mathbb{Q}$-Cartier, then 
    $\widetilde{X}$ is the blow up of $X$ along $C_1,C_2,C_3$ and  $\widetilde{C}$ is the preimage of $C$ under the morphism $\widetilde{X}\rightarrow X$.
    \item 
    if $K_X+\mathbf{nt}C$ is $\mathbb{Q}$-Cartier, then set $(\widetilde{X},\mathbf{nt}\widetilde{C}):=(X,\mathbf{nt}C)$
\end{itemize}
\end{definition}

The first step is to prove the following proposition.

\begin{proposition}
\label{prop:weight-change-from-t-to-nt-on-pair}
Let $(X,\mathbf{t}C)$ be a stable pair parametrized by a point in $
\overline{\mathrm{LM}}(\mathbb{P}^2,n)\setminus\{e\}$. Then $(\widetilde{X},\mathbf{nt}\widetilde{C})$ is also stable, and hence it is parametrized by a point in $
\overline{\mathrm{M}}_{\mathbf{nt}}(\mathbb{P}^2,n)$.
\end{proposition}

\begin{proof}
First, we argue that $K_{\widetilde{X}}+\mathbf{nt}\widetilde{C}$ is $\mathbb{Q}$-Cartier. By \cite[Theorem~5.7.2]{Ale15}, this is handled by performing the modification detailed in Definition~\ref{def:Q-Cartier-modification}. The ampleness of the log canonical divisor $K_{\widetilde{X}}+\mathbf{nt}\widetilde{C}$ follows from Lemma~\ref{lem:ampleness}. Log canonicity of the pair $(\widetilde{X},\mathbf{nt}\widetilde{C})$ is proved in Lemma~\ref{lem:lc-2-dim}.
\end{proof}

\subsubsection{Ampleness of the log canonical divisor}
\label{sec:ampleness-lc-divisor}

We first introduce some notation that we will retain for the rest of the section. Given a pair $(X,\sum_iB_i)$ with $X$ demi-normal, let $X=\cup_jX_j$ be its decomposition into irreducible components and $D_j\subseteq X$ the conductor divisor. Also, let $B_i^{(j)}$ be the restriction of $B_i$ to $X_j$.

\begin{lemma}
\label{lemma:ampleness-on-toric-variety-with-coefficient}
Let $X$ be a projective toric variety, $A$ an ample divisor on $X$, and $D$ a divisor on $X$. Let $c\in(0,1)\cap\mathbb{Q}$. Then, for $c\ll1$, the divisor $A+cD$ is ample.
\end{lemma}

\begin{proof}
By \cite[Theorem~6.3.20~(b)]{CLS11} we have that the effective cone of curves on $X$ is generated by the torus fixed curves, which we denote by $L_1,\ldots,L_m$. Therefore, by Kleiman's ampleness criterion \cite[Theorem~1.18]{KM98}, we have that $A+cD$ is ample provided $(A+cD)\cdot L_i>0$ for all $i$. Let $c_i\in(0,1)\cap\mathbb{Q}$ such that $(A+c_iD)\cdot L_i>0$, where $c_i$ exists because $A\cdot L_i>0$. If we define $c=\min\{c_1,\ldots,c_m\}$, then $(A+cD)\cdot L_i>0$ for all $i$.
\end{proof}

\begin{lemma}
\label{lem:ampleness}
Let $(X,\mathbf{t}C)$ be a stable pair parametrized by a point in $
\overline{\mathrm{LM}}(\mathbb{P}^2,n)$. Then, the divisor $K_{\widetilde{X}}+\mathbf{nt}\widetilde{C}$ is ample.
\end{lemma}

\begin{proof}
Let us start by considering the case where $K_X+\mathbf{nt}C$ is $\mathbb{Q}$-Cartier. By hypothesis, the divisor $K_X+\mathbf{t}C$ is ample. Moreover, by \cite[\S\,8]{KT06}, we have that $X$ a stable toric variety such that $K_X+\sum_{i=1}^3C_i\sim0$. Therefore, the divisor $\sum_{i=4}^nC_i$ is ample. Let $X=\cup_{j=1}^mX_j$, $D_j\subseteq X_j$, and $C_i^{(j)}$ be as at the beginning of \S\,\ref{sec:ampleness-lc-divisor}. By \cite[Proposition~1.2.16~(ii)]{Laz04} we have that $\sum_{i=4}^nC_i^{(j)}$ is ample on $X_j$.

We now show that $K_{X}+\mathbf{nt}C$ is ample. For this, it is enough to prove that for each $j$ the divisor
\begin{equation}
\label{eq:divisor-to-prove-is-ample}
K_{X_j}+D_j+(1-\epsilon)\sum_{i=1}^{3}C_i^{(j)}+\frac{1+\epsilon}{n-3}\sum_{i=4}^nC_i^{(j)}
\end{equation}
is ample. This can be rewritten as
\[
\left(K_{X_j}+\widetilde{D}_j+\sum_{i=1}^{3}C_i^{(j)}+\epsilon\sum_{i=4}^nC_i^{(j)}\right)+\left(\frac{1+\epsilon}{n-3}-\epsilon\right)\sum_{i=4}^nC_i^{(j)}-\epsilon\sum_{i=1}^{3}C_i^{(j)}.
\]
We already know that the first divisor in parentheses is ample. The second divisor is ample because $\sum_{i=4}^nC_i^{(j)}$ is ample by the discussion at the beginning of the proof, and its coefficient is positive provided $\epsilon$ is small enough. Finally, $\sum_{i=1}^{3}C_i^{(j)}$ is a divisor on the toric variety $X_j$ and $\epsilon$ is arbitrarily small, so we can conclude that the divisor \eqref{eq:divisor-to-prove-is-ample} is ample by Lemma~\ref{lemma:ampleness-on-toric-variety-with-coefficient}.

Next, we consider the case where $K_X+\mathbf{nt}C$ is not $\mathbb{Q}$-Cartier. So, consider the birational modification $\widetilde{X}\rightarrow X$. We only need to focus on the irreducible components $\widetilde{X}_j$ of $\widetilde{X}$, which are not isomorphic to the corresponding component $X_j \cong \mathbb{P}^1 \times \mathbb{P}^1$. Let $\widetilde{D}_j\subseteq\widetilde{X}$ be the conductor divisor. We need to check the ampleness of 
\(
K_{\widetilde{X}_j}+\widetilde{D}_j+\mathbf{nt}\widetilde{C}^{(j)}.
\)
For this, it suffices to check that the intersection with the torus fixed curves of $\widetilde{X}_j$ is positive. In particular, it suffices to check this for the exceptional curve $E$ of the blow up $\widetilde{X}_j\rightarrow X_j$ and the strict transforms $R_1,R_2$ of the distinct rulings passing through the blown up point.

We first consider $E$, which coincides with one of the curves among $\widetilde{C}_1^{(j)},\widetilde{C}_2^{(j)},\widetilde{C}_3^{(j)}$. We have that
\begin{align*}
&E\cdot\left(K_{\widetilde{X}_j}+\widetilde{D}_j+(1-\epsilon)\sum_{i=1}^3\widetilde{C}_i^{(j)}+\frac{1+\epsilon}{n-3}\sum_{i=4}^n\widetilde{C}_i^{(j)}\right)\\
={}&E\cdot\left(-C_1^{(j)}-C_2^{(j)}-C_3^{(j)}+(1-\epsilon)\sum_{i=1}^3\widetilde{C}_i^{(j)}+\frac{1+\epsilon}{n-3}\sum_{i=4}^n\widetilde{C}_i^{(j)}\right)\\
={}&1-(1-\epsilon)+\frac{1+\epsilon}{n-3}\cdot0=\epsilon>0.
\end{align*}

Finally, we consider $R_1$ (the argument for $R_2$ is analogous). We have that $R_1$ is an irreducible component of the conductor divisor $D_j$, and it can intersect transversely in one point exactly one or two curves among $\widetilde{C}_1^{(j)},\widetilde{C}_2^{(j)},\widetilde{C}_3^{(j)}$. Moreover, because of stability with respect to the weight vector $\mathbf{t}$, we have that $R_1$ has to intersect at least one curve among $\widetilde{C}_4,\ldots,\widetilde{C}_n$. It follows from these considerations that
\begin{align*}
&R_1\cdot\left(-C_1^{(j)}-C_2^{(j)}-C_3^{(j)}+(1-\epsilon)\sum_{i=1}^3\widetilde{C}_i^{(j)}+\frac{1+\epsilon}{n-3}\sum_{i=4}^n\widetilde{C}_i^{(j)}\right)
\\
\geq{}&-2+(1-\epsilon)\cdot2+\frac{1+\epsilon}{n-3}\cdot1=\frac{1+\epsilon(1-2(n-3))}{n-3},
\end{align*}
which is positive for $\epsilon\ll1$.
\end{proof}

\subsubsection{Study of log canonicity}

\begin{lemma}
\label{lem:lc-2-dim}
Let $(X,\mathbf{t}C)$ be a stable pair parametrized by a point in $
\overline{\mathrm{LM}}(\mathbb{P}^2,n)\setminus\mathrm{LM}(\mathbb{P}^2,n)$. Then $(\widetilde{X},\mathbf{nt}\widetilde{C})$ is log canonical.
\end{lemma}

\begin{proof}
With the notation introduced at the beginning of \S\,\ref{sec:ampleness-lc-divisor}, by hypothesis we have that the pair
\begin{equation}
\label{eq:irr-comp-stable-pair-LM-weights-dim2}
\left(X_j,D_j+\sum_{i=1}^{3}C_i^{(j)}+\epsilon\sum_{i=4}^nC_i^{(j)}\right)
\end{equation}
is log canonical and we aim to show that 
\begin{equation}
\label{eq:irr-comp-stable-pair-d-weights-dim2}
\left(\widetilde{X}_j,\widetilde{D}_j+(1-\epsilon)\sum_{i=1}^{3}\widetilde{C}_i^{(j)}+\frac{1+\epsilon}{n-3}\sum_{i=4}^n\widetilde{C}_i^{(j)}\right)
\end{equation}
is log canonical. First, we transform these conditions on the singularities into numerical ones.

Let $Z\subseteq X_j$ be a codimension $c$ closed subvariety contained in $r$ irreducible components of $D_j$ (so, $0\leq r\leq c\leq2$), $m_1$ divisors among $C_1^{(j)},C_2^{(j)},C_3^{(j)}$, and $m_2$ divisors among $C_4^{(j)},\ldots,C_n^{(j)}$. Since $X_j$ is smooth by \cite[Theorem~5.7.2]{Ale15}, the log canonicity of \eqref{eq:irr-comp-stable-pair-LM-weights-dim2} is equivalent to
\begin{equation}
\label{eq:lc-l-m-weight-dim2}
r+m_1+\epsilon m_2\leq c,
\end{equation}
We have two cases: the locus $Z$ intersects the center of $\widetilde{X} \to X$ or it does not. Both cases are studied independently. We start with the case where $Z$ does not intersect the non $\mathbb{Q}$-Cartier centers. In this case, to have a simplified notation, it will not be restrictive to work with $X$ and $C$ instead of $\widetilde{X}$ and $\widetilde{C}$.

Then, the log canonicity of \eqref{eq:irr-comp-stable-pair-d-weights-dim2} is equivalent to 
\begin{equation}
\label{eq:non-lc-d-weight-dim2}
r+(1-\epsilon)m_1+\frac{1+\epsilon}{n-3}m_2\leq c.
\end{equation}
Let us examine the possible cases. First, we suppose that $r=c$. In this case, \eqref{eq:lc-l-m-weight-dim2} implies that $m_1=m_2=0$. As a consequence, \eqref{eq:non-lc-d-weight-dim2} holds as well. So, we suppose that $r<c$. Then, the pair $(r,c)$ can be equal to
\[
(0,1),~(0,2),~(1,2).
\]
In the above cases, if $m_2\leq n-4$, then a direct check shows that the inequality \eqref{eq:non-lc-d-weight-dim2} is automatically satisfied (in some cases one may have to assume that $\epsilon$ is small enough). So, from now on we assume that $m_2=n-3$, which via \eqref{eq:lc-l-m-weight-dim2} implies that $0\leq m_1<c-r$.

The argument so far reduced the discussion to the following cases for the triple $(r,c,m_1)$:
\[
(0,1,0),~(0,2,0),~(0,2,1),~(1,2,0).
\]
In the second and third case we have that the inequality \eqref{eq:non-lc-d-weight-dim2} is satisfied. So, we only focus on $(0,1,0),(1,2,0)$. These two triplets do not satisfy \eqref{eq:non-lc-d-weight-dim2}, which seems to contradict our thesis. However, we now show that the geometry of $(X_j,R_j):=(X_j,\sum_{i=1}^3C_i^{(j)}+\epsilon\sum_{i=4}^nC_i^{(j)})$ does not allow for these two cases to appear. We first recall a few facts.

By \cite[Theorem~5.3.7~(2)]{Ale15}, each $(X_j,R_j)$ is isomorphic to the log canonical model of a line arrangement $(\mathbb{P}^2,K_j:=\sum_{i=1}^3H_{i,j}+\epsilon\sum_{i=1}^nH_{i,j})$ with finite automorphism group. Moreover, we have the following diagram of birational maps:
\begin{center}
\begin{tikzpicture}[>=angle 90]
\matrix(a)[matrix of math nodes,
row sep=2em, column sep=2em,
text height=1.5ex, text depth=0.25ex]
{&(\widetilde{X}_j,\widetilde{R}_j)&\\
(X_j,R_j)&&(\mathbb{P}^2,K_j)\\};
\path[->] (a-1-2) edge node[above left]{$g$}(a-2-1);
\path[->] (a-1-2) edge node[above right]{$f$}(a-2-3);
\path[dashed,->] (a-2-1) edge node[]{}(a-2-3);
\end{tikzpicture}
\end{center}
where the morphism $f$ resolves the non-log canonical singularities --- that is, the pair $( \widetilde{X}_j, \widetilde{R}_j) $ is log canonical --- and $g$ is the birational contraction which yields a stable pair. Under these birational transformations, the open subsets $U_j=X_j\setminus((\cup_{i\neq j}X_i)\cup(\cup_{i=1}^nC_i^{(j)}))$ and $V_j=\mathbb{P}^2\setminus\cup_{i=1}^nH_{i,j}$ are isomorphic. Moreover, given the choice of weights $\mathbf{t}=(1,1,1,\epsilon,\ldots,\epsilon)$, in $\mathbb{P}^2$ we must have that the first three lines $H_{1,j},H_{2,j},H_{3,j}$ are in general linear position. This comes from the theory in \cite[\S\,5.3]{Ale15}: the stable pair $(X,\mathbf{t}C)$ corresponds to a matroid tiling $\cup_j\mathrm{BP}_{V_j}$ of the \emph{$\mathbf{t}$-cut hypersimplex} $\Delta_{\mathbf{t}}(3,n)$ \cite[Definition~4.4.1]{Ale15}, and each matroid polytope $\mathrm{BP}_{V_j}$ has $(1,1,1,0,\ldots,0)$ as a vertex because this is the only integral vertex of $\Delta_{\mathbf{t}}(3,n)$. So, the line arrangements $(\mathbb{P}^2,K_j)$, which correspond to the polytopes $\mathrm{BP}_{V_j}$, have the lines $H_{1,j},H_{2,j},H_{3,j}$ in general linear position.

Let us consider the case $(r,c,m_1)=(0,1,0)$. Here, we have that the non-log canonical center $Z$ has codimension $1$ and it is not contained in a conductor divisor. Moreover, $m_1=0$ means that $Z$ is not contained in $C_1^{(j)},C_2^{(j)},C_3^{(j)}$. 
Then $Z$ is supported in the union of the $(n-3)$ lines $C_4^{(j)},\ldots,C_n^{(j)}$. Since $Z$ is of dimension $1$, then, $Z$ is equal to such a union. Therefore, using the above facts, we have that $K_j$ is supported on four lines which are in general linear position, that is, $H_{1,j},H_{2,j},H_{3,j},H_{4,j}=\ldots=H_{n,j}$. But this line arrangement in $\mathbb{P}^2$ is stable for the weight vector $\mathbf{t}$, so $(X,\mathbf{t}C)\cong(\mathbb{P}^2,K_j)$, which is parametrized by $e\in\mathrm{LM}(\mathbb{P}^2,n)$. This is against our hypotheses.

Finally, let us consider $(r,c,m_1)=(1,2,0)$. In this case, the non-log canonical locus is at the intersection of exactly one irreducible component of the conductor divisor $D_j$ and the $n-3$ divisors $C_4^{(j)},\ldots,C_n^{(j)}$. Because of the isomorphism $U_j\cong V_j$ and the fact that the planes $H_{1,j},H_{2,j},H_{3,j}$ are linearly independent, we have that the arrangement $(\mathbb{P}^2,\sum_{i=1}^3H_{i,j}+\epsilon\sum_{i=4}^nH_{i,j})$ is, up to automorphisms, as in Figure~\ref{fig:line-arrang-in-P2}~(a) or (b). The plane arrangement (a) is excluded because it has a positive dimensional family of automorphisms. So, the plane arrangement is as in (b). This is log canonical for the weight vector $\mathbf{t}$, which implies that $(X,\mathbf{t}C)=(\mathbb{P}^2,K_j)$, which contradicts the assumption that $(X,\mathbf{t}C)$ is not parametrized by $\mathrm{LM}(\mathbb{P}^2,n)$.

At this point, we are almost done. We recall to the 
tenacious reader still with us that the divisor $K_X+\mathbf{nt}C$ may not be $\mathbb{Q}$-Cartier. As discussed at the beginning of \S\,\ref{sec:iso-in-case-of-lines}, this issue can be resolved by considering $\widetilde{X}$ and $\widetilde{C}$ as in Definition~\ref{def:Q-Cartier-modification}. So, we consider the case in which $Z$ intersects the center of the birational modification $\widetilde{X}_j\rightarrow X_j$. By the discussion in Lemma~\ref{lem:Q-Cartier-modification}, we only have the case in which $X_j\cong\mathbb{P}^1\times\mathbb{P}^1$ and $\widetilde{X}_j$ is the blow up of $X_j$ at one of the torus fixed points.

We first argue that $Z$ equals the center of $\widetilde{X}_j\rightarrow X_j$. This is because $Z$ is not contained in any line with weight $\epsilon$; otherwise, the pair \eqref{eq:irr-comp-stable-pair-LM-weights-dim2} would not be log canonical. So, we have that $Z$ is a torus fixed point of $\mathbb{P}^1\times\mathbb{P}^1$. Let $\widetilde{Z}\subseteq \widetilde{X}_j$ be the preimage of $Z$ and let us check log canonicity along $\widetilde{Z}$. Let $V\subseteq\widetilde{Z}$ be the intersection of $r$ irreducible components of the conductor divisor $\widetilde{D}_j$, $m_1$ divisors among $\widetilde{C}_1^{(j)},\widetilde{C}_2^{(j)},\widetilde{C}_3^{(j)}$, and $m_2$ divisors among $\widetilde{C}_4^{(j)},\ldots,\widetilde{C}_n^{(j)}$. Let $c_V$ be the codimension of $V$ in $\widetilde{X}_j$. 
Then, instead of \eqref{eq:non-lc-d-weight-dim2}, we have to verify that
\[
r+(1-\epsilon)m_1+\frac{1+\epsilon}{n-3}m_2\leq c_V.
\]
If $V=\widetilde{Z}$, then the inequality above becomes
\[
0+(1-\epsilon)\cdot1+\frac{1+\epsilon}{n-3}\cdot0\leq 1,
\]
which is true. Now assume that $V\subseteq\widetilde{Z}$ is a torus fixed point. The inequality becomes
\[
1+(1-\epsilon)\cdot1+\frac{1+\epsilon}{n-3}\cdot0\leq 2,
\]
which is also satisfied.
\end{proof}

\begin{figure}
\begin{tikzpicture}
\draw [line width=1.pt] (1.5,4.3301270189221945)-- (-1.5,-0.8660254037844386);
\draw [line width=1.pt] (0.5,4.3301270189221945)-- (3.5,-0.8660254037844388);
\draw [line width=1.pt] (4.,0.)-- (-2.,0.);
\draw [line width=1.pt,dotted] (1.277350098112616,4.424870537968279)-- (0.2886751345948131,1.);
\draw [line width=1.pt,dotted] (0.7226499018873866,4.424870537968278)-- (1.7113248654051874,1.);
\draw (0.6,1.8) node[anchor=north west] {$\ldots$};
\draw [line width=1.pt,dotted] (9.96658749959701,1.988387706244832)-- (7.259393536879127,1.2704457663721431);
\draw [line width=1.pt,dotted] (9.689832914783025,2.4560194185472506)-- (7.753131178959709,0.4234819221240477);
\draw [line width=1.pt] (7.5,4.3301270189221945)-- (10.5,-0.8660254037844363);
\draw [line width=1.pt] (11.,0.)-- (5.,0.);
\draw [line width=1.pt] (5.5,-0.8660254037844386)-- (8.5,4.330127018922196);
\draw (7.442659779254485,1.35) node[anchor=north west] {$\rotatebox{-60}{\dots} $};
\draw [fill=black] (1.,3.4641016151377557) circle (3.5pt);
\draw [fill=black] (9.,1.7320508075688776) circle (3.5pt);
\draw (1.0,-1.5) node {(a)};
\draw (8.0,-1.5) node {(b)};
\end{tikzpicture}
\caption{In the proof of Lemma~\ref{lem:lc-2-dim}, line arrangements $(\mathbb{P}^2,\sum_{i=1}^3H_{i,j}+\epsilon\sum_{i=4}^nH_{i,j})$ corresponding to $(X_j,D_j+\sum_{i=1}^3C_i^{(j)}+\epsilon\sum_{i=4}^nC_i^{(j)})$. The solid lines represent $H_{1,j},H_{2,j},H_{3,j}$ and the dotted lines are $H_{4,j},\ldots,H_{n,j}$.}
\label{fig:line-arrang-in-P2}
\end{figure}
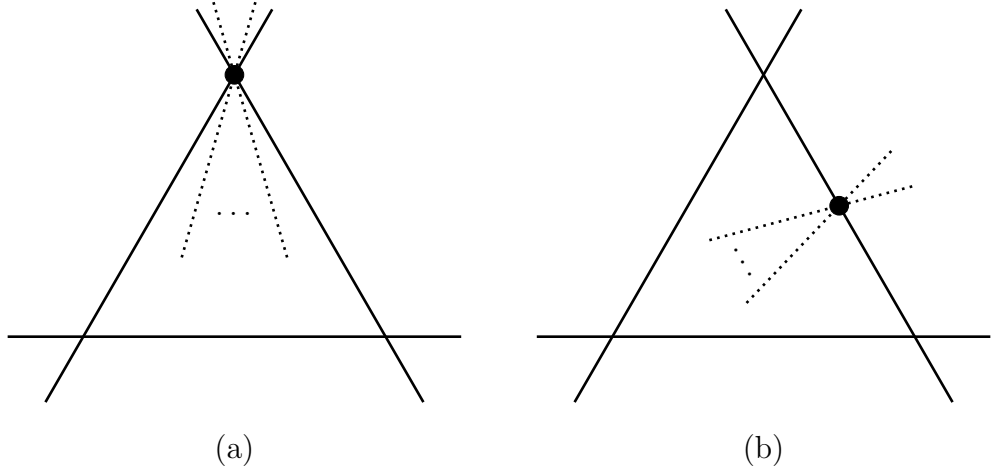


\begin{proof}[Proof of Theorem~\ref{thm:geometric-meaning-blow-up-identity-dim2}]
By the proof of Theorem~\ref{thm:locally-at-e-is-simple-blow-up}, we have a finite birational morphism
\[
\overline{\varphi}\colon\Bl_e\mathrm{LM}(\mathbb{P}^2,n)\rightarrow\mathcal{U}_{\mathbf{nt}}(\mathbb{P}^2,n).
\]
Our goal is to show that this extends to a finite birational morphism 
\[
\Bl_e\overline{\mathrm{LM}}(\mathbb{P}^2,n)\rightarrow\overline{\mathrm{M}}_{\mathbf{nt}}(\mathbb{P}^2,n).
\]
Let $x\in\Bl_e\overline{\mathrm{LM}}(\mathbb{P}^2,n)\setminus\Bl_e\mathrm{LM}(\mathbb{P}^2,n)=\overline{\mathrm{LM}}(\mathbb{P}^2,n)\setminus\mathrm{LM}(\mathbb{P}^2,n)$, which parametrizes a stable pair $(X,\mathbf{t}C)$. Let $R$ be a DVR with maximal ideal $\mathfrak{m}$ and let $K$ be its field of fractions. Let $g\colon\mathrm{Spec}(K)\rightarrow U:=\Bl_e\mathrm{LM}(\mathbb{P}^2,n)\setminus\mathbb{E}_n$ which extends to $\overline{g}\colon\mathrm{Spec}(R)\rightarrow\Bl_e\mathrm{LM}(\mathbb{P}^2,n)$ with $\overline{g}(\mathfrak{m})=x$. Let $f:=\varphi\circ g$ and let $\overline{f}$ be its unique extension to an $R$-point of $\overline{\mathrm{M}}_{\mathbf{nt}}(\mathbb{P}^2,n)$, which by continuity has to lie in $\overline{\mathrm{M}}_{\mathbf{nt}}(\mathbb{P}^2,n)\setminus\mathcal{U}_{\mathbf{nt}}(\mathbb{P}^2,n)$. Let us consider $y:=\overline{f}(\mathfrak{m})$. Then $y$ parametrizes the stable pair $(\widetilde{X},\mathbf{nt}\widetilde{C})$ (see Proposition~\ref{prop:weight-change-from-t-to-nt-on-pair}), regardless of the specific one-parameter family $g$ chosen. This implies that $\varphi$ extends as claimed by \cite[Lemma~3.18]{AET23}. The isomorphism with the normalization follows by the Zariski Main Theorem.
\end{proof}


\section{
The blow up of the Losev--Manin space and the Mori dream property
}
\label{sec:MDS}

We start with some preliminary results. 
Let $X$ be a normal variety with a finitely generated divisor class group. Additionally, suppose that $X$ has only constant invertible global regular functions. If $\Gamma$ is a finitely generated semigroup of Weil divisors on $X$, then we can consider the $\Gamma$-graded ring given by:
\begin{align*}
R(X,\Gamma):=\bigoplus_{D \in \Gamma}H^0( X, \mathcal{O}_X(D)),
\end{align*}
where $\mathcal{O}_X(D)$ is the divisorial sheaf associated to the Weil divisor $D$. If the classes of the Weil divisors generating $\Gamma$ span the vector space  $\mathrm{Cl}(X)_{\mathbb{Q}}$, then the ring $R(X,\Gamma)$ is called \emph{a Cox ring} of $X$. Although the definition of $R(X,\Gamma)$ depends on $\Gamma$, it is well known that whether it is finitely generated as an algebra over the base field is independent of the choice of $\Gamma$. A normal, projective, $\mathbb{Q}$-factorial variety $X$ with a finitely generated class group is called a \emph{Mori dream space} if it has a finitely generated Cox ring. A Cox ring of $X$ can be identified with a Cox ring for each of its small $\mathbb{Q}$-factorial modifications. Hence $X$ is a Mori dream space if and only if any given small $\mathbb{Q}$-factorial modification of $X$ is a Mori dream space. Moreover, the image of a Mori dream space is a Mori dream space, as shown by Okawa \cite{Oka16}.

\begin{theorem}[{\cite[Theorem~1.1]{Oka16}}]
\label{thm:Okawa}
Let $X,Y$ be normal, projective, $\mathbb{Q}$-factorial varieties and let $f\colon X \rightarrow Y$ be a surjective morphism. If $X$ is a Mori dream space, then so is $Y$.
\end{theorem}

\begin{remark}
It is widely expected by experts that, in general, the moduli space $\overline{\mathrm{LM}}(\mathbb{P}^d,n)$ is not $\mathbb{Q}$-factorial for $d\geq2$. We give an argument for this, communicated to us by Valery Alexeev, in the case of six lines in $\mathbb{P}^2$. The stable hyperplane arrangement $(X,\mathbf{t}D)$ in \cite[Figure~6.1~\#20]{Ale15} with weight vector $\mathbf{t}=(1,\epsilon,1,\epsilon,1,\epsilon)$ is parametrized by a point in $\overline{\mathrm{LM}}(\mathbb{P}^2,6)$. This point is the intersection of two boundary divisors generically parametrizing stable hyperplane arrangements as in \cite[Figure~6.1~\#11]{Ale15} (after appropriately labeling the lines), which correspond to two coarsenings of the subdivision associated to $(X,\mathbf{t}D)$. Since these two divisors intersect exactly at one point, we have that $\overline{\mathrm{LM}}(\mathbb{P}^2,6)$ is not $\mathbb{Q}$-factorial. In particular, also $\Bl_e\overline{\mathrm{LM}}(\mathbb{P}^2,6)$ is not $\mathbb{Q}$-factorial, and hence it cannot be a Mori dream space. However, one can still ask whether any of its \emph{$\mathbb{Q}$-factorializations}, as defined below, is a Mori dream space.
\end{remark}

\begin{definition}[{\cite[Definition~4.1]{Loh13}, see also \cite[Definition~6.25]{KM98}}]
Let $X$ be a normal projective variety. A \emph{$\mathbb{Q}$-factorialization} of $X$ is a proper birational morphism $f\colon Y\rightarrow X$ where $Y$ is a normal projective $\mathbb{Q}$-factorial variety and the exceptional set of $f$ has codimension greater than or equal to two in $Y$.
\end{definition}

We are ready to show the main result of this section. 
\begin{theorem}
\label{thm:Q-fact-not-MDS}
Any $\mathbb{Q}$-factorialization of the blow up $\Bl_e\overline{\mathrm{LM}}(\mathbb{P}^d,n)$ is not a Mori dream space for $n\geq d+9$.
\end{theorem}

\begin{proof}
We start by showing that there exists a dominant morphism 
\[
\Bl_e\overline{\mathrm{LM}}(\mathbb{P}^d,n)\rightarrow\Bl_e\overline{\mathrm{LM}}(\mathbb{P}^1,n-d+1).
\]
For convenience of notation, in this proof we change the usual order of the weights by moving the weights labeled by $3,\ldots,d+1$ at the end of the weight vector. So, we consider $\mathbf{t}=(1,1,\epsilon,\ldots,\epsilon,1,\ldots,1)\in\mathbb{Q}^n$ instead of the usual $(1,\ldots,1,\epsilon,\ldots,\epsilon)\in\mathbb{Q}^n$.

Consider the family of stable pairs on $\overline{\mathrm{LM}}(\mathbb{P}^d,n)$ from Theorem~\ref{thm:family-over-moduli-weighted-hyperplane-arrangements}:
\[
\left(\mathcal{X},\mathcal{B}_1+\mathcal{B}_2+\epsilon\left(\sum_{i=3}^{n-d+1}\mathcal{B}_i\right)+\mathcal{B}_{n-d+2}+\ldots+\mathcal{B}_n\right)\rightarrow\overline{\mathrm{LM}}(\mathbb{P}^d,n).
\]
Let us define $\mathcal{B}=\mathcal{B}_{n-d+2}\cap\ldots\cap\mathcal{B}_n$ and $\mathcal{P}_i=\mathcal{B}_i\cap\mathcal{B}$ for $i=1,\ldots,n-d+1$. Then, we can define by adjunction the following family on $\overline{\mathrm{LM}}(\mathbb{P}^d,n)$ of stable weighted rational curves:
\begin{equation}
\label{eq:family-of-curves-for-LM}
\left(\mathcal{B},\mathcal{B}_1|_{\mathcal{B}}+\mathcal{B}_2|_{\mathcal{B}}+\epsilon\left(\sum_{i=3}^{n-d+1}\mathcal{B}_i|_{\mathcal{B}}\right)\right)\rightarrow\overline{\mathrm{LM}}(\mathbb{P}^d,n).
\end{equation}
Therefore, there is an induced morphism $\overline{\mathrm{LM}}(\mathbb{P}^d,n)\rightarrow\overline{\mathrm{LM}}(\mathbb{P}^1,n-d+1)$. This morphism is dominant because every isomorphism class of pair $(\mathbb{P}^1,P_1+P_2+\epsilon\sum_{i=3}^{n-d+1}P_i)$ with distinct points $P_i$ can be realized in the image of the morphism \eqref{eq:family-of-curves-for-LM}.
This can be readily seen because for any configuration of points $P_1,\ldots,P_{n-d+1}$ on a line in $\mathbb{P}^d$ can be cut out by appropriate hyperplanes $H_1,\ldots,H_{n-d+1}$. The claimed morphism $\Bl_e\overline{\mathrm{LM}}(\mathbb{P}^d,n)\rightarrow\Bl_e\overline{\mathrm{LM}}(\mathbb{P}^1,n-d+1)$ will be obtained by extending the rational map at the top of the following commutative diagram:
\begin{center}
\begin{tikzpicture}[>=angle 90]
\matrix(a)[matrix of math nodes,
row sep=2em, column sep=2em,
text height=1.5ex, text depth=0.25ex]
{\Bl_e\overline{\mathrm{LM}}(\mathbb{P}^d,n)&\Bl_e\overline{\mathrm{LM}}(\mathbb{P}^1,n-d+1)\\
\overline{\mathrm{LM}}(\mathbb{P}^d,n)&\overline{\mathrm{LM}}(\mathbb{P}^1,n-d+1).\\};
\path[dashed,->] (a-1-1) edge node[above]{$\psi$}(a-1-2);
\path[->] (a-1-1) edge node[]{}(a-2-1);
\path[->] (a-2-1) edge node[]{}(a-2-2);
\path[->] (a-1-2) edge node[]{}(a-2-2);
\end{tikzpicture}
\end{center}
By construction, $\psi$ is dominant and it is defined away from the exceptional divisor $\mathbb{E}_n\subseteq\Bl_e\overline{\mathrm{LM}}(\mathbb{P}^d,n)$. We prove that $\psi$ extends over $\mathbb{E}_n$ by using the geometric interpretation of the two vertical blow ups in a neighborhood of the exceptional divisors.

Consider the weights given by
\begin{align*}
\mathbf{nt}&=\left(1-\epsilon,1-\epsilon,\frac{1+\epsilon}{n-d-1},\ldots,\frac{1+\epsilon}{n-d-1},1-\epsilon,\ldots,1-\epsilon\right)\in\mathbb{Q}^n,\\
\mathbf{a}&=\left(1-\epsilon,1-\epsilon,\frac{1+\epsilon}{(n-d+1)-2},\ldots,\frac{1+\epsilon}{(n-d+1)-2}\right)\in\mathbb{Q}^{n-d+1}.
\end{align*}
(We observe that $\mathbf{a}=\mathbf{nt}(1,n-d+1)$ according to Definition~\ref{def:weight_nt}.) By Theorem~\ref{thm:locally-at-e-is-simple-blow-up} and Lemma~\ref{lem:blow-up-of-Losev-Manin-for-points-in-P1} (after the end of this proof) we have the following isomorphisms:
\[
\Bl_e\mathrm{LM}(\mathbb{P}^d,n)\cong\mathcal{U}_{\mathbf{nt}}(\mathbb{P}^d,n)\subseteq\overline{\mathrm{M}}_{\mathbf{nt}}(\mathbb{P}^d,n)~\textrm{and}~\Bl_e\overline{\mathrm{LM}}(\mathbb{P}^1,n-1)\cong\overline{\mathrm{M}}_{\mathbf{a}}(\mathbb{P}^1,n-d+1).
\]
Let $x\in\mathbb{E}_n$. Let $R$ be a DVR with maximal ideal $\mathfrak{m}$ and $K$ its field of fractions. Let $g\colon\mathrm{Spec}(K)\rightarrow(\Bl_e\overline{\mathrm{LM}}(\mathbb{P}^d,n))\setminus\mathbb{E}_n$ such that its extension $\overline{g}\colon\mathrm{Spec}(R)\rightarrow\overline{\mathrm{M}}_{\mathbf{nt}}(\mathbb{P}^d,n)$ satisfies $g(\mathfrak{m})=x$. Let $f:=\psi\circ g$ and let $\overline{f}$ be its unique extension to an $R$-point of $\overline{\mathrm{M}}_{\mathbf{a}}(\mathbb{P}^1,n-d+1)$. Let us show that $y:=\overline{f}(\mathfrak{m})$ only depends on $x$ and not on the choice of $g$. This guarantees the extension of $\psi$ to $\Bl_e\overline{\mathrm{LM}}(\mathbb{P}^d,n)$ by \cite[Lemma~3.18]{AET23}.

Consider the family of weighted pointed stable rational curves 
on $\mathcal{U}_{\mathbf{nt}}(\mathbb{P}^d,n)$, which we denote by
\[
\left(\mathcal{S},(1-\epsilon)\mathcal{D}_1+(1-\epsilon)\mathcal{D}_2+\frac{1+\epsilon}{n-d+1}\left(\sum_{i=3}^{n-d+1}\mathcal{D}_i\right)+(1-\epsilon)\mathcal{D}_{n-d+2}+\ldots+(1-\epsilon)\mathcal{D}_n\right)
\]
Pulling it back along $\overline{g}$, we obtain a one-parameter family over $\mathrm{Spec}(R)$ whose central fiber
\[
\left(S,(1-\epsilon)D_1+(1-\epsilon)D_2+\frac{1+\epsilon}{n-d-1}\left(\sum_{i=3}^{n-d+1}D_i\right)+(1-\epsilon)D_{n-d+2}+\ldots+(1-\epsilon)D_n\right)
\]
gives the stable pair parametrized by $x$. This is as described in Definition~\ref{def:description-of-the-new-fibers} (see also Figure~\ref{fig:stable-line-arrangement-after-blow-up} for $d=2$). On the other hand, we can define the following family:
\[
\left(\mathcal{D},(1-\epsilon)\mathcal{D}_1|_{\mathcal{D}}+(1-\epsilon)\mathcal{D}_2|_{\mathcal{D}}+\frac{1+\epsilon}{(n-d+1)-2}\left(\sum_{i=3}^{n-d+1}\mathcal{D}_i|_{\mathcal{D}}\right)\right)\rightarrow\mathrm{Spec}(R),
\]
where $\mathcal{D}:=\mathcal{D}_{n-d+2}\cap\ldots\cap\mathcal{D}_n$. Its central fiber
\[
\left(D,(1-\epsilon)D_1|_{D}+(1-\epsilon)D_2|_{D}+\frac{1+\epsilon}{(n-d+1)-2}\sum_{i=3}^{n-d+1}D_i|_{D}\right)
\]
is also stable because $D$ consists of the transverse gluing of two copies of $\mathbb{P}^1$ at one point, one twig has two marked points $D_1|_{D},D_2|_{D}$ and the other twig has the remaining points $D_3|_{D},\ldots,D_{n-d+1}|_{D}$, which is stable. This is the stable pair parametrized by $y$. We can see that this only depends on $x$, and not on the particular choice of $g$.

By the above argument, we have a dominant morphism
\[
\Bl_e\overline{\mathrm{LM}}(\mathbb{P}^d,n)\rightarrow\Bl_e\overline{\mathrm{LM}}(\mathbb{P}^1,n-d+1).
\]
We know from \cite{HKL18} that $\Bl_e\overline{\mathrm{LM}}(\mathbb{P}^1,n-d+1)$ is not a Mori dream space for $n-d+1\geq10$ by \cite{HKL18} (this improved the previous lower bounds in \cite{CT15,GK16}). Therefore, we can conclude that any $\mathbb{Q}$-factorialization $Y\rightarrow\Bl_e\overline{\mathrm{LM}}(\mathbb{P}^d,n)$ is not a Mori dream space for $n\geq d+9$ by Theorem~\ref{thm:Okawa}.
\end{proof}

The next lemma, which is known to experts, is used in the proof of Theorem~\ref{thm:Q-fact-not-MDS}. It may be seen as the analog of Theorem~\ref{thm:geometric-meaning-blow-up-identity-dim2} for the case of $\mathbb{P}^1$, and we discuss its proof for completeness. 

\begin{lemma}
\label{lem:blow-up-of-Losev-Manin-for-points-in-P1}
Let $\mathbf{nt}=\left(1-\epsilon,1-\epsilon,\frac{1+\epsilon}{n-2},\ldots,\frac{1+\epsilon}{n-2}\right)\in\mathbb{Q}^n$. Then, the Hassett moduli space $\overline{\mathrm{M}}_{\mathbf{nt}}(\mathbb{P}^1,n)$ is isomorphic to $\Bl_e\overline{\mathrm{LM}}(\mathbb{P}^1,n)$.
\end{lemma}

\begin{proof}
Let $\mathbf{b}=(1-\epsilon,1-\epsilon,2\epsilon,\ldots,2\epsilon)\in\mathbb{Q}^n$. First, we observe that $\overline{\mathrm{LM}}(\mathbb{P}^1,n)=\overline{\mathrm{M}}_{\mathbf{b}}(\mathbb{P}^1,n)$ because the allowed incidences among the $n$ points are the same. Consider the reduction morphism $\rho_{\mathbf{b},\mathbf{nt}}\colon\overline{\mathrm{M}}_{\mathbf{nt}}(\mathbb{P}^1,n)\rightarrow\overline{\mathrm{M}}_{\mathbf{b}}(\mathbb{P}^1,n)$. Let $I=\{3,\ldots,n\}$ and $J=\{1,2\}$. Let $D_{I,J}\subseteq\overline{\mathrm{M}}_{\mathbf{nt}}(\mathbb{P}^1,n)$ be the boundary divisor generically paramentrizing the transverse gluing of two copies of $\mathbb{P}^1$ at one point, where one copy supports the points $\{P_3,\ldots,P_n\}$ while the other one supports the points $\{P_1,P_2\}$. The sum of the weights in $\mathbf{nt}$ indexed by $I$ sum up to $1+\epsilon>1$, but any proper subset of indices in $I$ has sum of the corresponding weights at most one. Then, by \cite[Remark~4.6]{Has03} we have that $\rho_{\mathbf{b},\mathbf{nt}}$ is the blow up of $\overline{\mathrm{M}}_{\mathbf{b}}(\mathbb{P}^1,n)$ along the image of the boundary divisor $D_{I,J}$, which is the point parametrizing the isorphism class of $(\mathbb{P}^1,\mathbf{b}P)$ with $P_3=\ldots=P_n$. Then, under the identification $\overline{\mathrm{LM}}(\mathbb{P}^1,n)=\overline{\mathrm{M}}_{\mathbf{b}}(\mathbb{P}^1,n)$, this stable pair is parametrized by the point $e$.
\end{proof}



\end{document}